\documentclass[12pt]{article}

\def\annotate{0} 

 \usepackage[latin1]{inputenc}  
 \usepackage{lmodern}
 \usepackage[T1]{fontenc}

\usepackage{amsmath,amsfonts,amssymb,amsthm,mathrsfs}
\usepackage{graphicx,xcolor}
\usepackage{stmaryrd}
\usepackage[numbers]{natbib}
\usepackage{cancel}
\usepackage[normalem]{ulem}
\usepackage{enumitem}

\RequirePackage[colorlinks=true, linkcolor = red,citecolor=blue,urlcolor=blue,backref=page]{hyperref}
\usepackage{cleveref}

\usepackage{tikz}
\usetikzlibrary{arrows.meta}

\usepackage[T1]{fontenc}
\usepackage{mlmodern}
\usepackage{eucal}
\usepackage{mathrsfs}
\usepackage[protrusion=true,expansion=true,auto=true,tracking=true]{microtype}

\newtheorem{theorem}{Theorem}

\newtheorem{example}{Example}

\newtheorem{lemma}[theorem]{Lemma}
\newtheorem{proposition}[theorem]{Proposition}
\newtheorem{remark}{Remark}

\newcommand{\tW}{ \widetilde{\mathsf W}}
\newcommand{\W}{ {\mathsf W}}
\newcommand{\C}{ {\mathsf C}}
\newcommand{\tC}{ \widetilde{\mathsf C}}

\newcommand{\R}{\mathbb{R}}
\newcommand{\Z}{\mathbb{Z}}

\newcommand{\Leb}{\mathcal{L}}
\newcommand{\tLeb}{\tilde{\mathcal{L}}}
\newcommand{\equlaw}{\stackrel{(d)}{=}}
\newcommand{\convlaw}{\stackrel{(d)}{\to}}
 \newcommand{\bn}{   \mathsf b_{n}}

\newcommand{\br}{\beta}
\newcommand{\1}{\mathbf{1}}
\newcommand{\md}{\mathrm{d}}
\newcommand{\tmu}{\tilde{\mu}}
\newcommand{\nupl}{\nu_{\textrm{Cl}}}

\newcommand{\E}{\mathbf  E}
\newcommand{\pr}{\mathbf  P}
\newcommand{\var}{\text{Var}}

\newcommand{\blue}[1]{\textcolor{blue}{#1}}

\newcommand{\REMOVE}[1]{}

 \if \annotate1

\newcommand{\obsolete}[1]{\textcolor{gray}{{#1}}}
\newcommand{\RL}[1]{\textcolor{blue}{#1}}

\newcommand{\dy}[1]{\textcolor{blue}{#1}}
\newcommand{\remove}[1]{{\color{grey} #1}}

\else

\newcommand{\obsolete}[1]{}
\newcommand{\RL}[1]{#1}
\newcommand{\dy}[1]{#1}
\newcommand{\remove}[1]{}
\renewcommand{\footnote}[1]{}

\fi

\title{Hyperuniformity and optimal transport of point processes}
\date{}
\author{Rapha\"{e}l Lachi\`eze-Rey\footnote{Lab. MAP5, Universit\'e Paris Cit\'e, and Inria Paris, France. Email: raphael.lachieze-rey@math.cnrs.fr} \, and D. Yogeshwaran\footnote{Theoretical Statistics and Mathematics Unit, Indian Statistical Institute, Bangalore, India. \\ Email: d.yogesh@isibang.ac.in}}

\begin{document}
\maketitle

\abstract{We examine optimal matchings or transport between two  stationary random measures. It covers allocation from the Lebesgue measure to a point process and matching a point process to a regular (shifted) lattice. The main focus of the article is the impact of hyperuniformity (reduced variance fluctuations in point processes) to optimal transport: in dimension $2$, we show that the typical matching cost has finite second moment under a  mild logarithmic integrability condition on the reduced pair correlation measure, showing that most planar hyperuniform point processes are $ L^2$-perturbed lattices. Our method also retrieves known sharp bounds in finite windows for neutral integrable systems such as Poisson processes, and also applies to hyperfluctuating systems. Further, in three dimension onwards, all point processes with an integrable pair correlation measure are $L^2$-perturbed lattices without requiring hyperuniformity. 
} \\

\noindent \textsc{Keywords:} Hyperuniformity, optimal transport, {matching}, structure factor, scattering intensity, random measures. \\

\noindent \textsc{MSC2020 Classification:} 60G55, 49Q22 ; 60D05, 60B15, 42A38.

\section{Introduction}

The celebrated AKT (Ajtai, Koml\'os and Tusn\'ady) theorem \cite{ajtai1984optimal} states that, if one draws $2n$ i.i.d.   uniform points $ X_{1},\dots ,X_{n},Y_{1},\dots ,Y_{n}$ in a two-dimensional cube with volume $n$, it is possible to pair each $ X_{i}$ with an exclusive $ Y_{\sigma (i)}$ for some permutation $ \sigma $  such that the typical length $|X_{1}-Y_{\sigma (1)}|$ is of the order $ \sqrt{\ln(n)}$ with high probability, and it is the best possible rate. The role of the dimension $ 2$ is very important, as for the same statement in dimension $ 1$, the typical distance is of order $ \sqrt{n}$, and in dimension $ d \geqslant 3$ it is of order $ 1.$ This result has known many refinements and generalisations such as replacing the total distance by weighted versions $ \sum_{} | X_{i}-Y_{\sigma (i)} | ^{p},p>0$,  drawing points with a non-uniform law on the cube, giving large deviations  \cite{fournier2015rate}, computing exact {asymptotic constants  \cite{ambrosio2019pde,huesmann2024asymptotics} etc}. Using the framework of optimal transport, these asymptotics are equivalent to the cost of optimal transport   {between} the measure {$\nu = \sum_{i=1}^n\delta_{Y_i}$ and the measure $\mu = \sum_{i=1}^n\delta_{X_i}$} on the cube. We refer the reader to \cite{shor1991minimax,talagrand2014upper,ledoux2022optimal} for surveys of such results and for a sampler of other recent developments, see \cite{barthe2013combinatorial,bobkov2021simple,goldman2023optimal}.

One may also study an {\it infinite volume} version of the AKT theorem by finding a {`good'} matching between infinite samples $ \mu  = \sum_{i}\delta _{X_{i}}$ and $ \nu  = \sum_{j}\delta _{Y_{j}}$,  {i.e. with minimal cost per unit volume,} where {$\mu$ and $\nu$ are independent unit intensity homogeneous Poisson processes. One may also consider the case of $\nu$ being lattice measures, or the Lebesgue measure in which case, matchings are replaced by transport maps.} It is shown in  \cite{holroyd2009poisson} that   {in the Poisson case} there is a {(randomized) bijection $T:\mathbb{Z} ^{d} \to  { \rm supp}  ( \mu)$ with the  typical distance $ X = \|T(0)\|$,  such that
\begin{equation}
\label{e:d2mom_Poisson}
\mathbf{E}\left(\frac{X^{d/2}}{1 + |\ln(X)|^{\gamma}}\right)<\infty \textrm{ for $d=1,2$, any $\gamma > 1$    and } \mathbf{E}(X^2) < \infty \textrm{ for $d \geqslant  3$},
\end{equation} 
and $\mathbf{E}(X^{d/2}) = \infty$ in $d = 1,2$ i.e., the bounds are sharp in these dimensions; see Remark \ref{r:HPPS_comparison}}. In higher dimensions, even finite exponential moments are proven {to exist} but we shall focus only on finite second moments; see \cite{hoffman2006stable,holroyd2009poisson,huesmann2013optimal,holroyd2022minimal}.

The starting point of this article is to replace $\mu$ and/or $\nu$ by an {arbitrary} stationary random point processes and consider matchings between $\mu$ and $\nu$ and equivalently transport {maps or plans between} the Lebesgue measure  {and} $\mu$. In specific planar cases such as zeros of Gaussian entire functions (GEF) or Ginibre point process, existence of good matchings have been shown.
 In particular, \cite{sodin2006random} shows exponential moments for the matching cost between zeros of GEF and lattice and \cite{jalowy2021wasserstein,prod2021contributions} prove respectively finite first and second moments for the transport {of Ginibre ensemble} to the Lebesgue measure.
Both these point processes are known to be hyperuniform ({also known as {\it super-homogeneous}}) i.e., the asymptotic {number variance (i.e. variance of the number of points on a large window)} grows more slowly than volume order, unlike the {Poisson point process \cite{torquato2003local,gabrielli2002glass}}. {Hyperuniform (HU)} processes appear in various contexts such as statistical physics, biology,  material sciences, random matrices, dynamical systems, numerical integration  \cite{torquato2018hyperuniform,coste2021order}.  A concise description coined by physicists about HU processes is ``global order and local disorder'', meaning that even though the particles are locally disordered, exhibit isotropy and asymptotic independence, they should behave at large scales as lattices, and should in some sense be arranged in a nice orderly manner. {A natural question is therefore whether such processes are well distributed in space, and we believe that this should be reflected by their transport properties, in the sense that a nicely spread point sample should be easy to transport to the Lebesgue measure, or equivalently to a regular lattice.} {Notable examples of hyperuniform point processes are  constructed as a good matching to a lattice; see \cite{peres2014rigidity,andreas2020hyperuniform,Lotz2023}.}

We see from  \eqref{e:d2mom_Poisson} that this property ({existence of a matching with finite second moment for typical cost}) is true in dimension $ d\geqslant 3$ for Poisson processes, and we will \dy{prove} that it extends to other processes, whether HU or not, as long as they are not hyperfluctuating. In dimension $ 2$, {we answer the question {on existence of 'good' matchings for HU processes in the affirmative}}. In other words,   {we show that for a  HU process $\mu$ in dimension $ d = 2$}, {identifying $ \mu $ with its support,} under mild logarithmic integrability conditions on the reduced pair correlation function, there is a $L^{2}$ process $ \{T(k);k\in \mathbb{Z} ^{2}\}$  {invariant under $ \mathbb{Z} ^{2}$ translations} such that {$\mu  = \sum_{k\in \mathbb{Z} ^{2}} \delta_{k + T(k)}$}, meaning it is a $L^{2}$ perturbed lattice. {Many known HU processes satisfy this condition, including the Ginibre point process and zeros of Gaussian entire function.} {An example of a HU process with infinite $L^1$-matching has been provided in \cite[Theorem 2(2)]{huesmann2024link} indicating the necessity of some additional integrability condition but whether logarithmic integrability is required or not is moot.  }
Our interest in $L^2$ perturbed lattices is also in part due to the following recent results on perturbed lattices due to \cite{DFHL}. Namely it was shown that $L^{2}$ perturbed lattices are HU but  for any $ \delta > 0$, there exist $L^{2-\delta}$ perturbed lattices which are not hyperuniform.

En route to our result for HU processes, we generalize the above bounds \eqref{e:d2mom_Poisson} for Poisson processes to stationary random point processes with integrable reduced pair correlation function (RPCM).  {This is to say that in magnitude, the matching bounds for point processes with number variance growing at most of volume-order cannot be worser than that for the Poisson process.} Our proofs proceed by bounding the transport or matching cost on large finite  windows via the bounds in \cite{bobkov2021simple} involving the
expected  squared Fourier-transform of the empirical measure of $\mu$. {These are strongly} related to the  {\it scattering intensity,}  which is known to converge to the Fourier transform of the RPCM, also called  {\it the structure factor}. These bounds work well for first and second moments but not higher moments. Vanishing of the structure factor at the origin is equivalent to hyperuniformity and quantifying the rate of vanishing on large windows and for small frequencies are the key estimates needed for our proof. Then we transfer these bounds to the infinite sample  {exploiting tightness and that the finite point processes (i.e., point process on finite windows) converge in law to the infinite sample in the vague topology}. This is in contrast to other proofs which use Whitney-type decomposition with error terms estimated via potential theory or PDE methods. Also the other works that derive bounds using the Fourier-analytic method, exploit 'weak independence' or 'mixing' of point processes to derive Poisson-like bounds but these do not exploit the vanishing of the structure factor. Our bounds in the finite sample case are explicit and {hold}  for all point processes without assumption of integrable reduced pair correlations.  Also in the infinite case, we give explicit bounds for matching when the logarithmic integrability condition is not satisfied by the hyperuniform point processes.  {Independently, \cite{butez2024} have obtained $p$-Wasserstein transport rates of HU processes on large windows under other additional assumptions on moments and concentration bounds.}
 {\cite{huesmann2024link} also recently obtain that HU processes are $L^{2}$ perturbed lattices, assuming finite Coulomb energy, another form of strong hyperuniformity.}
We give a more detailed comparison with the literature after Theorem \ref{thm:hu}.

{\bf Plan of the paper:} In the next subsection, we introduce some preliminaries on point processes. We state our main results and illustrate with examples in Section \ref{sec:mainresults}. We give more detailed proof outlines as well as comparisons with the existing literature. We also discuss the extension of our results to random measures in Section \ref{sec:extensionrm}. The proofs of our main results are in Section \ref{sec:proofs}. In that section, we also state the Fourier-analytic bounds for probability measures due to Bobkov and Ledoux as well as a general proposition that helps us to transfer linear transport cost between large finite samples to finite typical cost for the transport between infinite samples.

\subsection{Preliminaries}
We shall briefly introduce the required point process notions here and for more details, refer to \cite{hough2009zeros,kallenberg2017random,last2017lectures,
baccelli2020random} {and for the purposes of this article, \cite{coste2021order} also suffices.}

 {A {\it random measure} is a random element taking values in the space of locally-finite (or Radon) measures on $\R^d$ equipped with the evaluation $\sigma$-algebra.} A {\it point process} is a random measure  taking values in the space  {$  \mathscr   N$} of locally-finite counting measures on $\R^d$ equipped again with the evaluation $\sigma$-algebra. It is {\it simple} if a.s., $\mu(\{x\}) \in \{0, 1\}$ for all $x \in \R^d$, and   stationary or translation invariant if $\mu + x \overset{d}{=} \mu$ where $(\mu + x)(\cdot) := \mu(\cdot - x)$ is the shifted measure. {The set notation} is sometimes abusively applied to a simple point process by assimilating it to its support. Common examples of stationary point process are the shifted lattice $\sum_{z \in \Z^d}\delta_{z+U}$ for $U$ uniform in $[0,1]^d$, the homogeneous Poisson process, the infinite Ginibre ensemble, and zeros of Gaussian entire functions; see Section \ref{sec:examples} or \cite[Section 3]{coste2021order}.

 {Given a stationary point process, its first intensity measure $\mathbf{E}(\mu(\cdot))$ is proportional to the Lebesgue measure, and we shall assume that it is indeed the Lebesgue measure $ \Leb.$ Equivalently, {\bf we assume that all our point processes have unit intensity} i.e., $\mathbf{E}(\mu(\cdot)) = \Leb(\cdot).$ }

When $ \mu ,\nu $ are  {simple} point processes, {we abusively also denote by $ \mu ,\nu$ their supports}, and we call {\it matching} between $ \mu ,\nu $ a random one-to-one {map $ T:\nu \to \mu $.} If {$ \nu   = \Leb$, an {\it allocation scheme}  is a mapping $ T:\mathbb{R}^{d}\to     \mu$} that is also a transport map, meaning that a.s., $ \Leb(T^{-1}(A)) = \nu (A)$ for any $ A\subset \mathbb{R}^{d}$.
Whether we talk of matching or allocation scheme, we also require that $ T$ is  {\it invariant } (under simultaneous translations), i.e. the mass transported from $A$ to $B$ measurable subsets of $ \mathbb{R}^{d}$ satisfies for a shift $x\in  \mathbb R  ^{d}$ \begin{align*} \nu ( T^{-1}(A + x)\cap (B + x) ) \equlaw \nu ( T^{-1}(A)\cap B ).
\end{align*}{Without the simplicity assumption, such invariant mappings might not exist, which is why simplicity is always assumed in results about matching}. We remark that there always exist a (non-randomized) translation invariant allocation scheme $ T$ from $ \mathbb{R}^{d}$ to a simple point process $ \mu $ with  {unit} intensity, see Theorem 4 of \cite{hoffman2006stable}, where such an allocation is built with the Gale-Shapley algorithm on stable marriage (this particular procedure might be sub-optimal). The question is then to find a (different) $ T$ such that $ \|T(0)\|$ has the smallest  possible tail, or the largest possible cost function $ w$ such that for some $ T,$ $ \mathbf{E}(w(\|T(0)\|))<\infty $.

All the matchings or transport maps we consider are {\it randomized} matchings i.e., could depend on additional sources of  randomness apart from $\mu,\nu$. These concepts descend from the more general concept of  {\it coupling} between $ \mu $ and $ \nu $ in optimal transport, see Sections \ref{sec:ot_large} and \ref{sec:global}.
Given  two possibly dependent stationary point processes $ \mu ,\nu $,  Holroyd et al. \cite{holroyd2009poisson}  say that there is a matching between them with distribution function $ F$ if for some randomized invariant matching $ T$ between them
\begin{align}
\label{def:typical_matching_dist}
F(r) &:= (2\pi)^{-d} \, \mathbf{E} \, \dy{\nu\{x \in \Lambda_1 :  \|T(x)-x\|  \leqslant r \}}\,  , 
\end{align}where $\Lambda_1 = (-\pi,\pi]^d$.
 {  {To be slightly more general, we shall say that  {\it  there is an invariant coupling between $ \mu $ and $ \nu $ whose typical transport distance $ X$ has distribution function $ F$}
if there exists on some probability space a  random  measure $ M$   on $ \mathbb{R}^{d}\times \mathbb{R}^{d}$ such that $ (M(\cdot ,\mathbb{R}^{d}),M(\mathbb{R}^{d},\cdot ))$ has the same law as $ (\mu ,\nu )$, and}
%
\begin{align}
\label{def:typical-dist}F(r) 
& = (2\pi)^{-d} \, \mathbf{E} M(\{(x,y):x\in \Lambda _{1},\|x-y\|\leqslant r\})
,
\end{align}
\RL{and}
\dy{the random measure $M$ is {\em invariant} under simultanenous translations or equivalently {\em equivariant},  \RL{i.e.} $M(x + \cdot ,x + \cdot ) \equlaw M$   for all $x \in \R^d$.} This coupling is a matching if it can a.s. be represented by a one-to-one mapping $ T$, i.e. \RL{$ M = \sum_{x\in \nu'}  \delta _{\dy{(T(x),x)}}$ a.s., where $\nu'  = M(\mathbb{R}^{ d},\cdot ).$ Observe that matchings induced by invariant couplings are translation-invariant.}

There is also a formulation of typical matching or transport distance in terms of Palm distribution of $\mu$ or $\nu$ but we  \RL{do not require it here}; see \cite[Section 2]{holroyd2009poisson}. Also {when $\mu $ is simple} it is easy to see that $\mathbf{E}(w(X)) < \infty$ for some  function $ w:\mathbb{R}_{ + }\to \mathbb{R}_{ + }$ iff  
\begin{align}
\label{e:typical_cost_w}
\mathbf{E}\sum_{x\in \RL{ \nu' } \cap \Lambda _{1} }w(\|x-T(x)\|)<\infty .
\end{align}
The LHS above can be defined as the cost of transport between $\mu,\nu$ with respect to $w$; see \cite[(2.4)]{erbar2023optimal}.

A key descriptor of (unit intensity) point processes is the reduced pair correlation measure (RPCM) $ \beta $ of $ \mu $ {and as indicated by the name, it measures the asymptotic independence at the second order. Informally,}
\begin{align*}\beta (dx) = \mathbf{P}(\md x\in \mu \;|\;0\in \mu )-\mathbf{P}(dx\in \mu ).
\end{align*}
More formally, \dy{{\bf we assume implicitly that $\mu (A)$ has finite second moment for any bounded set $A$ whenever $ \beta $ is mentioned}}, in which case $ \beta $ is a signed measure \RL{over compact sets} characterised by the relation
\begin{align}
\mathbf{E}\sum_{x\neq y\in \mu }\varphi (x)\psi (y) &=  {\mathbf{E} \int \int \1[x \neq y]\varphi (x)\psi (y)\mu(\md x)\mu(\md y)}   \nonumber \\
& = \int \int \varphi (x)\psi (y)\md x \, \md y + \int \int \varphi (z)\psi (x + z) \md z \beta (\md x)\label{eq:def-beta}
\end{align}
for  non-negative compactly supported continuous functions $ \varphi ,\psi $. For instance, if $ \mu $ is a homogeneous Poisson process, or if $ \mu  = \Leb$, then $ \beta  = 0.$ Denote by $| \beta |$ the total-variation of $ \beta $, and say that $ \beta $ is integrable if $  | \beta  | (\mathbb{R}^{d})<\infty .$ We remark that existence of second moments guarantees that  $| \beta  |$ is locally integrable {and furthermore stationarity of $\mu$ implies that $\beta$ is an even function i.e., $\beta(\md x) = \beta(-\md x)$; for example, \dy{see \cite[Section 8]{DvJ08} for more on second order properties of random measures.}

 {A stationary point process $\mu$ is said to be {\it hyperuniform} (HU) if the {\it reduced variance} satisfies}
\begin{equation}
\label{e:sigman}
\sigma_{\mu}(n) := \frac{\var(\mu (\Lambda _{n}))}{\Leb(\Lambda _{n})} \to 0, \, \, \textrm{as $n \to \infty$},
\end{equation}
where $\Lambda_n = (-\pi n^{\frac{1}{d}},\pi n^{\frac{1}{d}}]^d$. Assuming integrability of the RPCM $\beta$, it is known that the choice of window shapes is not relevant i.e., instead of $\Lambda_n$ one can consider $n^{1/d}W$ for any convex set $W$ containing the origin in its interior; for example, see  \cite[Proposition 2.2]{coste2021order}.

The RPCM  is also central in the expression of the  {\it number variance }
\begin{align}
\label{e:variance_expression}
\var(\mu (A)) = \Leb(A) + \int_{A}\int_{z + A}\, \beta (\md x)\, \md z, \, \, A\subset \mathbb{R}^{d}.
\end{align}
 {Assuming integrability of RPCM $\beta$, if $\beta(\R^d) = \lim_{n\to \infty }\beta (\Lambda _{n}) = -1$, we have the variance reduction $ \var(\mu (\Lambda _{n})) = o(\Leb(\Lambda _{n}))$ and thus the process is HU.} The statistical study of a stationary point process $ \mu $ in the Fourier domain is often realised through its  {\it structure factor} defined via
\begin{align}
\label{e:structfact}
S(k) := 1 +\int_{}e^{-ik\cdot x}\br(\md x), \, \, k \in \R^d,
\end{align}
when it is well defined. This is the case if for instance the RPCM  $\beta$ is integrable. In this situation, HU can be characterized also as $S(0) = 0 $. {Some of these implications hold even without assuming integrability of $\beta$; see \cite[Theorem 3.6]{bjorklund2023hyperuniformity}.}

\RL{
Let us finally give a toy example \dy{of an HU process} with an integrable RPCM which is also a perturbed lattice; to our knowledge it was first emphasized in  \cite{cloak}. It is called  {\it cloaked} peturbed lattice, denoted by $ \nupl$, because the periodic component is hidden at the second order, or cloaked, by the vanishing of the Fourier transform of the $ U_{ k}$ on the grid. It helps make a bridge between disordered systems such as Poisson or determinantal, and perturbed lattices.
  \begin{example}
  \label{ex:iidpertlattice}
Consider the shifted lattice perturbed by  i.i.d variables $U_{ k},k\in \mathbb{Z} ^{ d}$ uniform  in $[0,1]^{ d}$ 
\begin{align*}
  \nupl  := \{k + U + U_{ k};k\in \mathbb{Z} \},
\end{align*}
where $U$ is also a uniform random variable in $[0,1]^{ d}$ and independent of the $U_k$'s. For $ \varphi ,\psi $ as above,
\begin{align*}
 \mathbf E \sum_{x\neq y\in   \nupl }\varphi (x)\psi (y) = &\sum_{k\neq m\in \mathbb{Z} ^{ d}}\int_{([0,1]^{ d})^{ 3}}\varphi  (k + u + v)\psi (m + u + w)\, \md u \, \md v \, \md w \\
  = &\int_{[0,1]^{ d}}\sum_{k, m\in \mathbb{Z} ^{ d}}\left(
\int_{[0,1]^{ d}}\varphi (k + u + v) \, \md v\int_{[0,1]^{ d}}\psi (m + u + w)\dy{\md w}
\right)\dy{\md u}\\
& \hspace{2cm}  - \int_{[0,1]^{ d}} \, \sum_{k\in \mathbb{Z} ^{d}} \left(
\int_{([0,1]^{ d})^{ 2}}\varphi  (k + u + v) \, \psi (k + u + w)\, \md v \, \md w
\right)du \\
&  {=  \int_{\mathbb{R}^{d}} \, \varphi \int_{\mathbb{R}^{ d}}\psi  -\int_{[0,1]^{ d}}\sum_{k\in \mathbb{Z} ^{ d}}\left(
\int_{([0,1]^{ d})^{ 2}} \, \varphi  (k + u + v) \, \psi (k + u + w)\md v \, \md w
\right)du}.
\end{align*}
 {Swapping the order of summation and integrals, the second term can be re-written as}
\begin{align*}
& \int_{[0,1]^{ d}}\sum_{k\in \mathbb{Z} ^{ d}}\left(
\int_{([0,1]^{ d})^{ 2}} \, \varphi (k + u + v) \, \psi (k + u + w)\md v \, \md w
\right)du \\
& = \int_{\R^d \times \R^d \times \R^d} \varphi (z+v) \, \psi(z+v+w) \, \1[v \in [0,1]^d, v + w \in [0,1]^d] \md z \, \md v \, \md w \\
& = \int_{\R^d \times \R^d}\varphi (z) \, \psi(z+w) \, \Leb([0,1]^d \cap ([0,1]^d - w)) \, \md z \, \md w,
\end{align*}
\dy{  and so from the characterisation  \eqref{eq:def-beta}, we see  that the RPCM of the i.i.d. perturbed lattice $\nupl$ is $ \beta(\md x) = \Leb([0,1]^d \cap ([0,1]^d - x)) \, \md x$.} Thus the RPCM is supported by $ B(0,2\sqrt{d})$ and clearly integrable. Furthermore, the obvious matching $T(k) = k + U + U_k, k \in \Z^d$ from $\Z^d$ to $\mu$ has typical matching distance $X \leq 2\sqrt{d}$ a.s.. \dy{Also, we know that $\nupl$ is hyperuniform; for instance, see \cite{bjorklund2024hyperuniformity}.}
\end{example} }

\section{Main results}
\label{sec:mainresults}
In this section, we state our main results. Firstly, we state our global matching result {(Theorem \ref{thm:std})} for stationary processes with integrable RPCM which shows that {generic point processes {in dimension $d \geq 3$} cannot have worser matching cost bounds than that of the Poisson process}. Further, assuming hyperuniformity, we show in {Theorem \ref{thm:hu}} that the bounds can be improved {in dimensions $d = 1,2$}. We also discuss our results in the context of other results, outline our proof ideas and provide examples of point processes verifying the above bounds. {The proofs  proceed via bounds for matching or transport cost between a point process and Lebesgue measure on large boxes. This is the content of Theorems \ref{thm:Wbd-infinite} and \ref{thm:Wbd-infinite-HU}, which are of independent interest and can also be viewed as an equivalent formulation of good transport between point processes.} The proofs of all these main theorems - Theorems \ref{thm:std}, \ref{thm:hu}, \ref{thm:Wbd-infinite} and \ref{thm:Wbd-infinite-HU} - are deferred to Section \ref{sec:proofs}. In Section \ref{sec:extensionrm}, we comment on extension to random measures.

{We begin by stating our global matching bound that shows that the infinite volume version of AKT bounds for the Poisson process as in \eqref{e:d2mom_Poisson} hold for most point processes}.
\begin{theorem}
\label{thm:std}
 Let $ \mu ,\nu $ be independent {simple} stationary processes  with unit intensity and having integrable RPCM.  There exists a translation-invariant matching $T$ between $\mu$ and $\nu$ such that the typical matching distance $X$ (\dy{see  \eqref{e:typical_cost_w}}) satisfies the following bound
\begin{align*}
\mathbf{E}w( X  )<\infty
\end{align*} for
\begin{align*}
w(x)=
\begin{cases}
\frac{ \sqrt{x}}{1+ | \ln(x) | ^{\gamma }}$  if $d= 1\\
\frac{ x}{1+ | \ln(x) | ^{\gamma }}$  if $d= 2\\
x^{2}$  if $d\geqslant 3.
 \end{cases}
\end{align*}
{for any} $\gamma > 1$. {The above bounds also hold if $\nu$ is the shifted lattice.}
\end{theorem}
%
%
\begin{remark}
\label{r:HPPS_comparison}
The bounds for $d=1,2$ are obtained following the methods of \cite[Theorem 1(ii)]{holroyd2009poisson} and are stated here for completeness.  {  We} can show that there exists a matching from $\nu$ to $\mu$  \RL{whose typical distance $ X$ satisfies}%
\begin{equation}
\label{e:M0-inf_matching}
 \pr(X \geqslant  r) \leqslant cr^{-d/2} \sqrt{\sigma_{\mu}(r) + \sigma_{\nu}(r)}, \, \, r \geq 1,
\end{equation}
for some finite $ c$ and with
 $\sigma_{\cdot}(r)$ as defined in \eqref{e:sigman}; {see the proof of Theorem \ref{thm:std} \RL{at Section  \ref{sec:proofs_thms12}} for more details on this derivation.}

Thus, in the integrable RPCM case, $\sigma $ is bounded and so we have that $\E(X^{\alpha}) < \infty$ for all $\alpha < d/2$ and in $d =1,2$, we have $\E(w (X)) < \infty$ for $w $ as above.   So, our Theorem \ref{thm:std} improves this bound in $d  = 3,4$ and in $d \geq 5$, \cite[Theorem 1(ii)]{holroyd2009poisson} (i.e., \eqref{e:M0-inf_matching}) yields better bounds.
\end{remark}

%

\REMOVE{\sout{In the above theorem, in dimension $  1$ and $ 2$, $ w$ is a  {\it continuity modulus}, i.e. it is a continuous non-decreasing function on $ \mathbb{R}_{ + }$ satisfying $ w(x + y)\leqslant w(x) + w(y)$ (see Proposition  \ref{p:cont_mod}){, ensuring triangle inequality, and therefore the applicability of the method of \cite{bobkov2020transport}}. For all the results of this paper, $ w$ is either a continuity modulus, or the square function $ w(x) = x^{2},$ \blue{in which case many matching or transport cost bounds can be accessed via PDE techniques.} }}

We now turn to hyperuniform point processes, i.e. we assume $\sigma(r) \to 0$ as $r \to \infty$. From \cite[Lemma 1.6]{nazarov2012correlation} (see also \cite[Theorem 2A]{beck1987irregularities} and \cite{bjorklund2024hyperuniformity}), we know that $\sigma(r) \geqslant  Cr^{-1}$ for some $ C>0$. Under the assumption of $|x||\beta(\md x)|$ being integrable,     it is known that $\sigma(r) = \Theta(r^{-1})$ \cite[Proposition 2]{martin1980charge}. {Thus \eqref{e:M0-inf_matching} can at the best yield
$$ \mathbf{E}\Big( \frac{X^{\frac{d+1}{2}}}{1 + |\ln(X)|^{\gamma}} \Big) < \infty,$$
for any $\gamma > 1$,} and in $d  = 2$, this still does not give finite second moments. We now show that {finite second moments holds} in $d=2$ under a weaker logarithmic integrability condition on the RPCM and for completeness, we state the above bound for $d =1$.  {The logarithmic integrability holds for many HU processes including the Ginibre process and zeros of Gaussian entire  functions. See Section \ref{sec:examples} for more examples  {and the comparison with literature paragraph on the next page about the logarithmic integrability condition.}
\begin{theorem}
\label{thm:hu}
 {Let $ \mu $ be a HU process with an integrable RPCM $\beta$.}
\begin{enumerate}
\item In $d = 2$, suppose $\beta$ further satisfies
\begin{align*}\int_{\mathbb{R}^{2}}\mathbf{1}[ | x | >1]\ln( | x | )  | \beta  | (\md x)<\infty.
\end{align*}
Then there exists a $\Z^2$-translation invariant matching $T$ from $\Z^2$ to $\mu$ with typical distance $X$ such that $\mathbf{E}(X^{2})<\infty  $ i.e., $ \mu $ is said to be a planar $ L^{2}$-perturbed lattice.

The same results hold when $\mathbb{Z}^{2}$ is replaced by an independent stationary HU process $ \nu $ satisfying the same conditions as $\mu$.

\item In $d = 1$, suppose \dy{the number variance $  \textrm{Var}\left(\mu (\Lambda_n)\right), n \geq 1$ is uniformly bounded}  {and let $\nu$ be an independent simple point process satisfying the same assumptions as $\mu$.  Then there exists a translation invariant matching $T$ from $\mu$ to $\nu$ with typical distance $X$ such that}
$$\mathbf{E}\left(\frac{X}{1 + |\ln(X)|^{\gamma}}\right)<\infty \textrm{ for any $\gamma > 1$.}$$
\end{enumerate}
\end{theorem}

%
%



%
\RL{
\begin{remark}  Like for Theorem \ref{thm:std}, the proof for $ d = 1$ borrows the approach of \cite{holroyd2009poisson}; see Remark  \ref{r:HPPS_comparison}.
\end{remark}}
Regarding the last statement in Item (1),    it can be obtained by building matchings \dy{$ T:\mathbb{Z}^{2} \to \mu,T':\mathbb{Z}^{2} \to \nu$} on a product probability space, then \dy{$T \circ (T')^{-1}$} is a matching between $\mu$ and $\nu$ with the same properties; in this regard the result
involving $ \mathbb{Z} ^{2}$ is stronger, {even though lattice or shifted lattice doesn't have an integrable RPCM. \dy{However, we first prove the bounds in the case of independent point processes with integrable RPCM and then use our toy example of \RL{a cloaked }perturbed lattice from Example \ref{ex:iidpertlattice} to obtain an invariant matching from $\Z^2$.} It seems here that the independence assumption between $\mu$ and $\nu$ is superfluous but  in the dependent case, the respective matchings of $ \mu ,\nu $ are built on a new
probability space, and it is not clear how they can be recombined on the same probability space by maintaining the same dependency. We also remark that the logarithmic
integrability condition in Item (1) implies that $\ln(n) \, \sigma_\mu(n) = o(1)$.}

\RL{ Regarding point (2), boundedness of the variance in dimension $ 1$ holds true in several examples; see for instance  \cite{AGL} who study models where this assumption is satisfied and give several examples, such as the 1D log gas.}

\paragraph{{Proof Ideas.}} For the proofs of Theorems  \ref{thm:std} (for $d \geq 3$) and \ref{thm:hu}, we derive transport rates ({i.e., cost of an allocation scheme}) between the point process restricted to a finite box and the restriction of the Lebesgue measure to that box using Fourier-analytic bounds due to Bobkov and Ledoux \cite{bobkov2021simple}; see Theorems \ref{thm:Wbd-infinite} and \ref{thm:Wbd-infinite-HU}. More specifically, consider the restricted sample $ \mu_{n}= \mu\1_{\Lambda _{n}}$, where $ \Lambda _{n} = (-\pi n^{1/d} ,\pi n^{1/d}]^{d} $ , and its renormalisation $ \tilde \mu _{n}= \frac{ n}{N}\mu _{n}$ (where $ N = \mu (\Lambda _{n})$),  for some `nice'  cost function $ w$. One obtains the linear cost
\begin{align*}
\mathbf{E} \widetilde{\mathsf C } _{w}(\tilde \mu _{n},\tLeb_{n}) \leqslant c \, n,
\end{align*}
 {where $\widetilde{\mathsf C } _{w}(\tilde \mu _{n},\tLeb_{n})$ is the cost with respect to $w$ of the
optimal {transport or allocation scheme} under the Toroidal metric on $\Lambda_n$ between $\tmu_n$ and $\tLeb_{n} = \Leb(\Lambda _{ n})^{ -1}\Leb\1_{ \Lambda _{ n}}$; 
see Section
\ref{sec:ot_large} for precise definitions. A similar bound is obtained for transport rates between $\tilde \nu_n$ (a suitably normalized version of $\nu$ in $\Lambda_n$ \dy{with total mass $n$}) and $\tLeb_n$. Using triangle inequality\dy{, one can obtain a coupling between $\tilde \mu _{n}$ and $\tilde \nu_n$ which is then extended to the infinite space via relative compactness arguments and such that $\tmu_n, \tilde \nu_n$ converge to $\mu,\nu$ in the vague topology. Finally, using a result from \cite{erbar2023optimal}, we show that the coupling between $\mu,\nu$ is realized by a matching.} This is done in Proposition \ref{p:matching_limit} in Section \ref{sec:global}, thereby yielding that there is a matching $ T:\nu \to \mu $  {of the limiting infinite measures}  such that
\begin{align*}
\mathbf{E}\int_{\Lambda _{1}}w(\|x-T(x)\|)\mu (\md x)<\infty
\end{align*}
which indeed corresponds to the definition of the typical distance  \eqref{def:typical-dist}.}  \dy{We allow for the possibility that the matching and the point processes $\mu,\nu$ may need to be defined on a new probability space.}

\RL{In dimensions $ d = 1,2$, the bounds on $n^{-1}\mathbf{E} \widetilde{\mathsf C } _{w}(\tilde \mu _{n},\tLeb_{n})$ from Theorem  \ref{thm:Wbd-infinite} below are diverging and do not allow us to derive the result of Theorem \ref{thm:std}, which is why we use the method of \cite{holroyd2009poisson}.}
   However, we mention that our finite volume bounds in Theorem \ref{thm:Wbd-infinite} match with those for i.i.d. points in all dimensions and also applies to the non-integrable RPCM case. \dy{The same comment would also apply to the case $d = 1$ in the setting of Theorem \ref{thm:Wbd-infinite-HU}.}

As is evident from our description, our method is non-constructive and yields only randomized matchings in the terminology of  \cite{hoffman2006stable}. We do not address existence of factor matchings or transport maps as in \cite{hoffman2006stable,chatterjee2010gravitational,
nazarov2007transportation,huesmann2024transportation} or their geometric properties or precise constants in the asymptotics as in \cite{ambrosio2019pde}.

In all the paper, $$  c = c(\mu ;\nu ;w)$$ denotes a finite constant whose value might vary from line to line, and might depend on the laws of some random measures $ \mu ,\nu $ and a cost function $ w.$ No effort is taken on bounding $ c$, but its value might be tracked from the articles  {\cite{bobkov2020transport,bobkov2021simple}}.

\paragraph{{Comparison with literature.}}
\dy{We now comment on our approach in the context of other matching or transport results for} hyperuniform processes. Both in the specific examples such as \cite{sodin2006random,jalowy2021wasserstein,prod2021contributions}  and the general results of \cite{butez2024} the authors use also strong moment assumptions on the point process. The proof of \cite{sodin2006random} proceeds by constructing a Whitney-type partition with respect to a modified metric and a potential-theoretic lemma; see also \cite{sodin2010uniformly}. \cite{jalowy2021wasserstein} also uses potential-theoretic estimates. Such potential-theoretic estimates work well for zeros of Gaussian entire  functions or Ginibre ensemble. \cite{prod2021contributions,butez2024} rely upon sub-additivity with suitable Whitney-type decompostions and the crucial estimates of the error in sub-additive arguments are achieved using PDE or functional-analytic tools as for combinatorial optimization problems over i.i.d. points in  \cite{ambrosio2022quadratic,goldman2021convergence,goldman2023optimal}. Both the potential-theoretic or PDE approach involve estimating solutions of a certain distributional Poisson equation. Our approach stemming from \cite{bobkov2020transport} is very different to these and involves only estimates based on the pair correlation function of the point processes. While the proof strategy is relatively simpler and works in greater generality, our methods cannot give higher-moment estimates unlike in the above articles (when the assumptions are sufficiently strong). It is important to mention that \cite{huesmann2024link} also recovers Theorem \ref{thm:hu} by showing that a hyperuniform process for which the normalized variance (i.e., \eqref{e:sigman}) goes to $0$ at a logarithmic speed, the Coulomb energy is finite, and the latter implies finite mean $2$-Wasserstein distance to Lebesgue measure (or equivalently to the lattice); see Remark 1.8 therein.

 Interestingly, an analogous result to Theorem \ref{thm:hu} in the deterministic setting which is existence of bounded Lipschitz matching, was recently proven in \cite{koivusalo2024sharp} under a similar logarithmic condition on discrepancy and also using Fourier-analytic methods. \dy{Also in specific examples, it is quite likely that the matching bounds can be improved significantly; for example, very recently \cite{elboim2025optimal} have proven that perturbed lattices satisfy matching tail bounds corresponding to their hole probabilities.}

\subsection{Optimal transport of large samples}
\label{sec:ot_large}

The best suited framework to explain our results and proofs is that of optimal transport. Given two measures $ \mu ,\nu $ on $ \mathbb{R}^{d}$ with same (non-zero) mass, and $ w:\mathbb{R}_{ + }\to \mathbb{R}_{ + }$ nondecreasing, define the $ w$-transport cost as
\begin{align*}
  \mathsf C_{w}(\mu ,\nu ) : = \inf _{M}\int_{\mathbb{R}^{d}\times \mathbb{R}^{d}}w(\|x-y\|) \, M(\md x,\md y)
\end{align*}
where $ M$ is a coupling ({or transport plan}) between $ \mu $ and $ \nu $, i.e. {$M$ is a measure on $\R^d \times \R^d$ with same mass as $\mu$ and $\nu$ such that  }
\begin{align*}M(\cdot ,\mathbb{R}^{d}) = \mu ,\;M(\mathbb{R}^{d},\cdot ) = \nu .
\end{align*}
{By  \cite[Theorem 1.7]{santambrogio2015optimal}, if $\mu,\nu$ have same finite total mass and $ c$ is lower semi-continuous   then the infimum is indeed reached by some $ M$, called the  {\it optimal transport plan} or  {{\it coupling}. Since we will work with continuous cost functions (modulus of continuity or $x^2$) existence of optimal coupling is always guaranteed for us.} If furthermore $ \mu ,\nu $ are supported on some cube $ \Lambda $, let $ d_{\Lambda }(x,y)$ be the toric distance on $ \Lambda $, and the corresponding cost is
\begin{align}
\label{e:toric_cost}
 {\widetilde {\mathsf C}_{w}(\mu ,\nu)} = \widetilde {\mathsf C}_{w}(\mu ,\nu ;{\Lambda}) : = \inf_{M} \int_{\Lambda ^{2}}w(d_{\Lambda }(x,y)) \, M(\md x,\md y).
\end{align}
If $\mu = \nu = 0$ (i.e., the null measure), {by definition} $ \mathsf C_{w}(\mu ,\nu ) = \widetilde {   \mathsf C} _{w}(\mu ,\nu ) = 0$. {We take the toric distance as we rely on Fourier decomposition on the torus, but in the large window asymptotics,
\RL{the mean transport cost for the toric and Euclidean metric coincide under our assumptions.}
 We often drop the dependency on the domain $\Lambda$ when it is obvious.}

A prominent case in optimal transport is the power cost function $ w_{p}(t) = t^{p},p>0$, and the corresponding costs {$\mathsf W_{p}^{p}:  =  \mathsf C_{w_{p}}, \widetilde{\mathsf W}_{p}^{p} =  \mathsf{ \widetilde{C}_{w_p}}  $}. For $ p\geqslant 1,$ the functionals $   \mathsf W_{p}$ and $ \widetilde{   \mathsf W}_{p}$ are distances, called $ p$-Wasserstein (toric) distance. There are several exhaustive monographs on optimal transport; in the current work we only work on Euclidean spaces, for which \cite{santambrogio2015optimal} provides the necessary background, see also \cite{erbar2023optimal} for more on optimal transport of random measures.

 {Our standard assumption in Theorems \ref{thm:std} and \ref{thm:hu} is about an integrable RPCM $\beta$. However in our upcoming finite sample bounds, we shall explicitly quantify the dependence of matching or transport cost on the tail of $\beta$, which might not be integrable. We now introduce two quantities  measuring the tail of   $\beta$.} \dy{Recalling $ \Lambda _{n} = (-\pi n^{1/d} ,\pi n^{1/d}]^{d},$} define
\begin{align} \label{eq:bn} \bn: = \dy{1 + |\beta|(\Lambda_n + \Lambda_n)} = 1 +  \int_{\Lambda _{2^{d}n} } | \beta  | (\md x),
\end{align}
\begin{align}
\label{eq:def-eps}
\varepsilon (t):=  \int_{\mathbb{R}^{d}}\min(1,{  t| x | }) \, |\beta|(\md x),t>0.
\end{align}
{Note that under integrability of   $\beta$, $\bn$ is a bounded sequence and   $\varepsilon(t)$ appears implicitly in the logarithmic integrability condition in Theorem \ref{thm:hu} \dy{and upcoming Theorem \ref{thm:Wbd-infinite-HU}(2)}.}

 {We now state two theorems which are finite-window versions of our Theorems \ref{thm:std} and \ref{thm:hu} respectively and more importantly are key elements in the proof of those theorems. The first upcoming theorem (Theorem \ref{thm:Wbd-infinite}) can be seen as a generalisation of the AKT Theorem where the i.i.d. sample is replaced by a dependent sample with a random number of points and expectedly the second theorem (Theorem \ref{thm:Wbd-infinite-HU}) improves upon the bounds under hyperuniformity assumption.}

More formally, we consider independent stationary point processes $ \mu ,\nu $, and upper bound their transport cost on the large cube $\Lambda_n$. Using the Fourier-analytic method, which goes through estimating Fourier-Stieljes transform of the restricted samples $ \tilde \mu _{n}: = \frac{n}{N}\mu\1_{\Lambda _{n}}$, where $ N = \mu (\Lambda _{n})$ {is the number of points}, we obtain the following bounds.

\begin{theorem}
\label{thm:Wbd-infinite}
Let $ \mu $ be a stationary point process with unit intensity. Denote by $ \tilde  \mu _{n}= \frac{n}{N} \mu\1_{\Lambda _{n}}$   {and $\tLeb_n := \frac{1}{(2\pi)^d} \Leb \dy{\1_{\Lambda _{n}}}$}, the renormalised samples with mass $n$ on $ \Lambda _{n}.$ {When $N = 0$, set $\tmu_n = n\delta_0$.} Then we have the following bounds on the squared Wasserstein costs
\begin{align}
\label{e:W22munLn}
{ \mathbf{E} { \widetilde{ \mathsf W}}_{2}^{2}(\tilde \mu _{n},\tLeb_{n};{\Lambda_n}) \, } &\leqslant c \, \alpha_{2} (n),\\
\sqrt{\mathbf{E}
  { \widetilde {\mathsf W}}_{p}^{{2p}}(\tilde \mu _{n},\tLeb_{n};{\Lambda_n})} &\leqslant c \, \alpha_{p} (n), p\in {(0,1]} \nonumber \\
\end{align}
where $\alpha_2, \alpha_p$ are given by
\begin{align}
\label{eq:macro-rate-std}\alpha _{2}(n)\, := \, &  { \bn }\times
\begin{cases}
n^{2}$  if $d= 1\\
n\ln(n)$  if $d= 2\\
n$  if $d\geqslant 3,
 \end{cases}\\
\label{eq:macro-rate-std-p}
\alpha_{p} (n) \, := \, &
  \begin{cases}
  {\sqrt{ \bn}}\,  n  \, n^{p-1/2}$ if $d = 1, p \in {(1/2,1]} \\
  {\sqrt{ \bn} \, n \, \ln(n)}$ if $d = 1, p = 1/2 \\
  {\sqrt{\bn} \, n \, \sqrt{\ln(n)}}$ if $d = 1, p \in (0,1/2) \\
  {\sqrt{\bn^p} \, n \, \sqrt{\ln(n)^p}}$  if $d = 2\\
  {\sqrt{\bn^p}} \, n$  if $d \geq 3. \\
   \end{cases},
   \end{align}
   %
%
%
\end{theorem}
\begin{theorem}
\label{thm:Wbd-infinite-HU}
Consider the same setting as in Theorem \ref{thm:Wbd-infinite}. Further, assume that $\beta $ is integrable and  $ \mu $ is HU, i.e. $\beta (\mathbb{R}^{d})= -1$. Then, we have the following bounds for Wasserstein cost:
\begin{enumerate}
\item
\begin{align}
\label{eq:rate-hu-Wp} {\mathbf{E}
  { \widetilde {\mathsf W}}_1^2(\tilde \mu _{n},\tLeb_{n};{\Lambda_n}) \leq c \, n \, \alpha_2(n)} \quad ; \quad \mathbf{E} { \widetilde{ \mathsf W}}_{2}^{2}(\tilde \mu _{n},\tLeb_{n};{\Lambda_n}) \,  & \leqslant c \, \alpha_{2} (n),
  \end{align}
where $\alpha_2(n)$ can be chosen as follows for any $c_0 > 0$,
\begin{align}
\label{eq:rate-hu}
\alpha_2(n) :=  c \, n \,
\big( 1+\int_{n^{-1/d}}^{c_{0}}\varepsilon(r) \, r^{d-3} \, { \md r} \big).
\end{align}
\item {In $d = 2$ if $\beta$ satisfies}
\begin{align*}\int_{\mathbb{R}^{2}}\mathbf{1}[ | x | >1] \, \ln( | x | ) \, | \beta  | (\md x)<\infty,
\end{align*}
then we have that $\mathbf{E} {\widetilde{ \mathsf W}}_{2}^{2}(\tilde \mu _{n},\tLeb_{n};{\Lambda_n}) \,   \leqslant c \, n.$
\end{enumerate}


%
\end{theorem}
The proofs are in Section  \ref{sec:proof_Wbd-infinite}. The choice  $\tilde \mu_n = n\delta _{ 0}$ for $N = 0$ is in some sense the worst choice but for our results one could have chosen any other measure with mass $n$. {Note the linear rate of growth for $W_2^2$-transport cost in Theorem \ref{thm:Wbd-infinite} for $d \geq 3$ and in Theorem \ref{thm:Wbd-infinite-HU}(2) for $d=2$ respectively. These will be crucial to the proofs of Theorem \ref{thm:std} and \ref{thm:hu}. Though we cannot obtain linear-rate of growth in $d = 1$ for HU processes, \dy{we could improve the bounds from Theorem \ref{thm:Wbd-infinite} as follows. If $\int_{\R}|x| \, |\beta| (\md x) < \infty$ then we have that $\mathbf{E} { \widetilde{ \mathsf W}}_{2}^{2}(\tilde \mu _{n},\tLeb_{n};{\Lambda_n}) \,  \leqslant c \, n \, \ln(n).$ Since this is not used in our main results, we do not \RL{provide a proof.}}

{Theorem \ref{thm:Wbd-infinite} for Wasserstein distance is consistent with those known for i.i.d. points or mixing point processes or point processes generated via mixing Markov chains; see for example \cite[Section 7]{fournier2015rate}}, \cite[Section 5]{bobkov2021simple}, \cite[Section 2]{bobkov2024correction}, \cite[Theorem 2]{borda2021empirical} and \cite[Proposition B.1]{clozeau2024annealed}}. Like with all our other main theorems, there are two significant differences. Firstly, our bound is for random points and for general point processes, no mixing-type assumption is required. In case the point processes have integrable reduced pair correlations (which is only about the second order marginal and weaker than the mixing assumptions in the cited articles), we recover the same bound as \cite{bobkov2021simple,borda2021empirical,fournier2015rate} for $2$-Wasserstein distance. Secondly {in Theorem \ref{thm:Wbd-infinite-HU}}, we exhibit the particular behaviour of HU processes  whereas the above papers do not distinguish between hyperuniform and non-hyperuniform point processes as this concept is not related to mixing.

{As noted earlier, the above two theorems rely on the idea  of Bobkov and Ledoux  to combine Fourier-analytic bounds with a smooth convolution \cite{bobkov2020transport,bobkov2021simple,bobkov2024correction} and as in their works, it is possible to state the finite-volume bounds for more general costs even though the analogue in infinite-volume is unclear.}
Consider general $w$-transport costs $\tilde{C}_w$ where $w:\mathbb{R}_{  + }\to \mathbb{R}_{  + }$ is a continuity modulus, i.e. non-decreasing, continuous and sub-additive ($w(x+y) \leq w(x) + w(y)$). Observe that $w(x) = x^p$, $p \in [0,1]$ are continuity moduli. Though we use only these choices in our results, for completeness we state the general bound here.

It involves a non-decreasing function $q : [1,\infty) \to [0,\infty)$  such that
$$A_{q}: = \sum_{k=0}^{\infty}\frac{1}{q(2^k)} < \infty.$$
A standard choice is $q(k) = \ln(2k) \, \ln \ln(3k)^{\gamma}$ for some $ \gamma >1$. With $q$ as above and for any $t_0 > 0$, we have the following bound for $\mu_n$ as in Theorem \ref{thm:Wbd-infinite}.
$$ \mathbf{E}\widetilde{\C}_{w}^{2}(\tilde \mu _{n},\tLeb_{n};\Lambda_n) \leq c \, n \, \bn \, \sum_{k=1}^{t_0} k^{d-1} q(k) \, w(n^{1/d}k^{-1})^{2} +  cn^2 w(n^{1/d}t_0^{-1})^{2}.$$
This can be derived from Theorem \ref{thm:BL} and scaling properties as for $\mathsf{W}_{p}^{2p}$ in the proof of Theorem \ref{thm:Wbd-infinite}.
Further if we assume that $\mu$ is HU as in Theorem \ref{thm:Wbd-infinite-HU}, then the above bound can be improved as follows:
$$ \mathbf{E}\widetilde{\C}_{w}^{2}(\tilde \mu _{n},\tLeb_{n};\Lambda_n) \leq c \, n \, \bn \, \sum_{k=1}^{t_0} k^{d-1}q(k)\, \varepsilon(kn^{-1/d})\, w(n^{1/d}k^{-1})^{2} +  cn^2 w(n^{1/d}t_0^{-1})^{2},$$
where $\varepsilon$ is as in \eqref{eq:def-eps}.

\begin{remark}\label{rk:var-eps-repulsive}

Even if the RPCM is a favored tool in spatial statistics, the variance is a more wide-spread statistical indicator, and the two are closely related, especially for hyperuniform processes. We have in this case {the variance for the number of points {$N = \mu(\Lambda_n)$}}
\begin{align*}
\var(N) = -n\int_{\mathbb{R}^{d}}\min( 1,n^{- 1/d}\|x\|)\beta (\md x)
\end{align*}
(see  \eqref{eq:var-gamma-beta}) and in particular, with the triangle inequality, $ \var(N)  \leqslant cn\varepsilon  (n^{- 1/d}))$ for some $ c<\infty $. If the reverse inequality {(i.e., $ \var(N)  \geqslant cn\varepsilon  (n^{- 1/d}))$)} were to hold as one would expect in many cases, it would allow us to equivalently state our main theorems in terms of $\var(N)$ instead of $\beta$. However, the possibility of oscillating $ \beta $ prevents us from doing so in the general case. Under the assumption that $ \beta $ has negative sign, we \dy{show in Lemma \ref{lm:var-bd} that}
\begin{align*} \textrm{Var}(N)\asymp n\varepsilon (n^{- 1/d});
\end{align*}
\dy{The negative sign} assumption means that in some sense the $ 2$-point correlation function is repulsive. Some determinantal point processes, or more generally, {the so-called {\it repulsive} point processes} satisfy this assumption because their $ n$-point correlation function is repulsive at all orders (and everywhere); {see \cite{blaszczyszyn2014clustering}}.
\end{remark}

\subsection{Examples of planar HU processes}
\label{sec:examples}

\subsubsection{Determinantal processes}

Determinantal points processes (DPPs), introduced in the context of quantum mechanics, have gained popularity as many classes of essential models in random matrix theory, statistical physics, combinatorics and others have proven to be determinantal, see \cite{hough2009zeros}. In the Euclidean context, a simple   point process $ \mu $ with unit intensity on $ \mathbb{R}^{d}$ is determinantal with  kernel
 $ K:\mathbb{R}^{d}\times \mathbb{R}^{d} \to \mathbb{C}$ if for every $ k\in \mathbb{N}^{*}$,
\begin{align*}\rho _{k}(x_{ 1},\dots ,x_{k}): = \det((K(x_{i},x_{j}))_{ 1\leqslant i,j\leqslant k})
\end{align*}
is the $ k$ point correlation function of $ \mu $, i.e.
for any non-negative $ \varphi :(\mathbb{R}^{d})^{k}\to \mathbb{R}$, we have
\begin{align*}\mathbf{E}\left(
\sum_{x_{ 1},\dots ,x_{k}\in \mu }^{\neq }\varphi (x_{ 1},\dots ,x_{k})
\right) = \int_{}\rho _{k}(x_{ 1},\dots ,x_{k})\varphi(x_{ 1},\dots ,x_{k})\, \md x_{ 1} \, \dots \, \md x_{k} ,
\end{align*}
where the sum runs over $ k$-tuples of pairwise distinct points. Note that not all functions $ K$ give rise to a DPP and in particular we will require that $K$ is Hermitian and positive definite; see  \cite[Section 4.5]{hough2009zeros} for unicity and existence questions. A {degenerate} example is the unit intensity homogeneous Poisson process for which $ \rho _{k} \equiv  1,$ i.e. $K (x,y) = \mathbf{1}_{\{x = y\}}.$ {A non-trivial and prominent example} is the (unit intensity) infinite Ginibre ensemble on $\mathbb{R}^{2}\approx  \mathbb C $, defined by
\begin{align*}K(x,y) = e^{x\bar y-\frac{1}{2}|x|^2-\frac{1}{2}|y|^2};x,y\in  \mathbb C .
\end{align*}
One has \begin{align*}\rho _{2}(x,y) = K(x,x)K(y,y)-K(x,y)K(y,x) = 1-\exp(- | x-y | ^{2}) = 1- | K(0,x-y) | ^{2},
\end{align*}
and $\rho_{2}$ only depends on the difference $ x-y$, as it turns out  the Ginibre ensemble is indeed stationary; \dy{see for example \cite[Proposition 5.6.6]{baccelli2020random}}. It is at the crossroads of combinatorics, statistical physics, and random matrices, as it is also the planar Coulomb gas with inverse temperature $ 2$, and it also arises as the infinite limit of the process formed by the eigenvalues of a $ n\times n$ matrix with i.i.d.   standard complex Gaussian entries (see  \cite{hough2009zeros}).

Using the explicit definition of $\rho_2$ and \eqref{eq:def-beta}, we obtain that for a stationary DPP,
\begin{align*}\md x \, \beta (\md z) =(\rho _{2}(x,x + z)-1)\md x\md z =   - | K(0,z) | ^{2}\md x\md z.
\end{align*}
As outlined in Remark  \ref{rk:var-eps-repulsive}, the negativity of $ \beta $ means that the process is repulsive at the second order (one may prove that it is actually repulsive at any order).  The hyperuniformity of a unit intensity stationary DPP $ \mu $ hence means that $ \int_{} | K(0,z) | ^{2}\md z =  1$. HU DPPs can also be characterised as those {with kernel $ K$ defining a $ L^{2}$ operator whose} spectrum is contained in $ \{0, 1\} $  \cite{hough2009zeros}. We have the following corollary of Theorem \ref{thm:hu}:\begin{theorem}
 A planar stationary determinantal process with  kernel $ K$ is a $L^{2}$-perturbed lattice  ({i.e., $\ \mathbf{E}(X^2) < \infty$ for typical matching distance $X$}) if it is hyperuniform, i.e.
\begin{align*}\int_{}| K(0,z) | ^{2}\md z =  K(0,0)^{2},
\end{align*}
and if
\begin{align*}\int_{}\ln( \| z \| )| K(0,z) | ^{2}\md z<\infty .
\end{align*}
Furthermore, the bounds on the transport costs for restricted samples of Theorem \ref{thm:Wbd-infinite-HU} hold  with {$\alpha_2(n) = cn$ for some constant $c$.}
 \end{theorem}
Similar bounds for the expected $   \mathsf W_{ 1}$ (resp. $   \mathsf W_{2}$) cost for the finite Ginibre ensemble have been shown by  \cite{jalowy2021wasserstein} (resp.  \cite{prod2021contributions}).
%
%

\subsubsection{Coulomb gases and other hyperuniform processes}

Hyperuniform processes have been intensively studied in the physics literature, and classified into three classes in  \cite{torquato2018hyperuniform}, depending on whether the restricted number variance $\sigma (r)$ satisfies $ \sigma (r)\sim r^{- 1}$  (class I), $\sigma (r)\sim r^{- 1}\ln(r)$ (class II), or $ \sigma (r)\sim r^{-\alpha },\alpha \in (0, 1)$ (class III).  {As far as we know,} no process of interest seems to fall in the boundary class where $r^{-\alpha }\ll \sigma (r) \ll 1$, where the  {log}-integrability condition might not hold, and the mere existence of such processes is not trivial. All three classes mean (at least) a polynomial decay {of $\beta$}, hence should in principle exhibit sufficient asymptotic independence to satisfy the logarithmic condition of Theorem \ref{thm:hu}(1). Nevertheless, we are not aware of existing mathematical proofs of such statements.

 An exception is the zero set $ \mu $ of the planar Gaussian  {entire } function, for which the exact fast decay of the RPCM can be inferred from the explicit formulas of the second order correlation measure of  \cite[Section 3]{hough2009zeros}. {Thus combined with Theorem \ref{thm:hu}, this} would yield that such zero sets are $L^2$-perturbed lattices; it has actually already been shown in  \cite{nazarov2007transportation} that this is the case, and the authors gave a strong concentration bound on the typical matching distance with an explicit and elegant transport plan to the Lebesgue measure built from the gravitational flow generated by $ \mu $. {The stable matching of lattice (or some 'nice' HU point processes) to a Poisson process (or more generally DPPs) of larger intensity produces a hyperuniform point process whose reduced pair correlation satisfies the log integrability condition; see \cite[Proposition 8.2]{andreas2020hyperuniform} {and \cite{Lotz2023}}. But in this case, the construction itself yields a good matching with the lattice.}

A particular focus has been put in both the mathematical and physics literature on 2D Coulomb gases, also called one component plasma {(OCP)} or Jellium, of which the Ginibre ensemble is a particular case at temperature $1/2$; see \cite{torquato2018hyperuniform,leble,serfaty}.
Hyperuniformity of the 2D OCP has been proved by Lebl\'e  \cite{leble} after having been conjectured for a long time. More precisely he proved that $ \sigma (r)\leqslant c\ln(r)^{-0.6}$, but physicists  estimate that the second order asymptotic independance is actually much stronger, meaning that the discrepancy is minimal, i.e. $ \sigma (r)\sim r^{- 1}$, falling in the class I of hyperuniform processes, but it remains for now out of reach for a mathematical treatment.

\subsection{Extension to Random measures}
\label{sec:extensionrm}

{We now comment on extension of our results to random measures. We shall suggest the modified definitions and then appropriate changes to our main theorems. In particular, these involve pointers to upcoming results and proofs and hence pre-suppose a cursory reading of the corresponding result statements.} Also for some of the basic notions of random measures, we again refer the reader to \cite{baccelli2020random,kallenberg2017random,last2017lectures}

Suppose $\mu$ is a unit intensity stationary random measure. Then $\E \mu(A) = \Leb(A)$. Further, {relying again on the Campbell-Mecke formula}, we can define the RPCM $\beta$ and the diagonal intensity $\lambda_D \geq 0$ as follows. Let $g : (\R^d)^{2} \to \R_+$ be a measurable function.
\begin{equation}
\label{e:campbell1} \E \int_{(\R^d)^2} g(x,y) \mu^2(\md (x,y)) =  \int_{\R^d \times \R^d} g(x,x+z){\big(\beta + 1\big)}(\md x)\md z  + \lambda_D \int_{\R^d} g(z,z)  \, \md z.
\end{equation}
 {In case $\mu$ is a simple point process, $\lambda_D = 1$ and often extension of point process results to random measures need to suitably adapt to the fact that $\lambda_D \neq 1$. In case the random measure is diffuse (i.e., has not atoms a.s.), then $\lambda_D = 0$;} \dy{for example, see }\cite{krishnapur2024stationary}.

Assume that $\lambda_D < \infty$ and $\beta$ is integrable. Thus, the variance can be written as
$$\var(\mu (A)) = \lambda_D \Leb(A) + \int_{A}\int_{z + A}\beta (\md x) \, \md z, \, \, A\subset \mathbb{R}^{d}.$$
We can define the structure factor as
\begin{align}
S(k) := \lambda_D \, + \, \int_{}e^{-ik\cdot x}\br(\md x), \, \, k \in \R^d.
\end{align}
%

HU in the case of stationary random measures can be characterized as $\int \beta = -\lambda_D$ and equivalently $S(0)= 0$ as for simple point processes. Since Lemmas \ref{lm:var-bd} and \ref{lm:approx-sf} also hold for random measures $\mu$ as well ({again adapting to the fact that $ \lambda _{D}\neq 1$ for a general random measure}), one can again use Theorem \ref{thm:BL} to prove Theorems \ref{thm:Wbd-infinite} and \ref{thm:Wbd-infinite-HU} for a random measure $\mu$. Equipped with Theorems \ref{thm:Wbd-infinite} and \ref{thm:Wbd-infinite-HU}, we can again use Proposition \ref{p:matching_limit} to obtain results on typical transport cost for infinite samples. Note that since $\mu$ is a random measure, one cannot talk about matchings but rather transport plans or couplings. In other words, we have that for a random measure with integrable RPCM, there exists a coupling $M$ of $\mu$ and $\nu$ such that
$$ \E \int_{\Lambda_1 \times \R^d}w(\|x-y\|)M(\md x,\md y) < \infty,$$
for weight function $w$ as in Theorem \ref{thm:std}. Furthermore, under the logarithmic integrability condition and hyperuniformity as in Theorem \ref{thm:hu}, we have that in $d = 2$, there exists a coupling $M$ of $\mu$ with the Lebesgue measure on $\R^d$ such that
$$ \E \int_{\Lambda_1 \times \R^d}\|x-y\|^2 M(\md x,\md y) < \infty,$$
Similarly, analogues of other claims in Theorem \ref{thm:hu} also hold.

\section{Proof of main theorems - Theorems \ref{thm:std}, \ref{thm:hu}, \ref{thm:Wbd-infinite} and \ref{thm:Wbd-infinite-HU}}
\label{sec:proofs}

{In this section, we prove our main theorems. In Section \ref{sec:fourierprobbounds}, we restate the Bobkov-Ledoux Fourier-analytic bound for probability measures adapted to our purposes. We use this with variance and empirical structure factor estimates to prove Theorems \ref{thm:Wbd-infinite} and \ref{thm:Wbd-infinite-HU} in Section \ref{sec:proof_Wbd-infinite}. In Section \ref{sec:global}, we state the proposition that helps us to transfer matchings or transport on large boxes to the typical cost of infinite matching or transport. Finally, we use this to conclude the proof of Theorems \ref{thm:std} and \ref{thm:hu} in Section \ref{sec:proofs_thms12}.}

 \subsection{Fourier-analytic bounds for distance between probability measures}
 \label{sec:fourierprobbounds}
We now state Fourier-analytic bounds for probability measures. These are essentially contained in \cite[Proposition 2]{bobkov2021simple} and \cite[Theorem 1.1]{bobkov2020transport}. Similar Fourier-analytic bounds are used in \cite{brown2020wasserstein} and \cite{borda2021empirical}. See \cite{steinerberger2021wasserstein} for more applications of such bounds.

For a finite measure $P$ on {$\Lambda _{1} =  (-\pi,\pi]^{d}$}, we denote its {\it Fourier-Stieltjes} transform as {
$$ f_{P}(m) := \int_{\Lambda _{1}}e^{\dy{-i\, m \cdot x}}P (\md x), \, \, m \in \Z^d.$$
For notational convenience, we shall not always indicate that $f_P$ is defined over $\Z^d$ only.}  {$P$ is a probability measure and} $X$ is a random variable with distribution $P$, we use $f_X$ to denote is {\it Fourier-Stieltjes} transform $f_P$. {Also by linearity, we write $f_{P} - f_Q$ as $f_{P-Q}$ in case of two probability measures $P,Q$. Also recall from Section \ref{sec:ot_large}, the notion of transport costs $\widetilde{\mathsf W}_p,\widetilde{\mathsf C}_{w}$ under the toroidal distance.} {Recall that $q : [1,\infty) \to (0,\infty)$ is a non-decreasing function with
$$A_{q}: = \sum_{k=0}^{\infty}\frac{1}{q(2^k)} < \infty,$$
and the common choice being $q(x) = \ln(2x)\ln(\ln(3x))^{\gamma}$ for $\gamma > 1$.}

\begin{theorem}[Bobkov-Ledoux]
\label{thm:BL}
      \label{thm:bd-finite-sample}Let $Q$ be the uniform {probability distribution} on $\Lambda _{1}$. {There is a constant $c \in (0,\infty)$ (possibly depending on $d$)} such that for any probability measure $P$ on $\Lambda _{1}$ and all ${t_0} > 1$, we have that %
\begin{align}
\label{eq:BL-tildeW2}
  \widetilde{\mathsf W}_{2}^{2}(P ,Q ; \Lambda_1)& \leqslant  \,c {\sum_{0< \| m \| \leqslant t_0}\|m\|^{-2}| f_{P}(m) | ^{2}}+c \, t_0^{-2}.
\end{align}
Suppose \dy{$Q$ is any probability distribution on $\Lambda_1$ and $w$ is a continuity modulus on the torus. 
 Then {for a constant $c \in (0,\infty)$ (possibly depending on $d$)} such that for any probability measure $P$ on $\Lambda _{1}$ and all ${t_0} > 1$, we have that}
{
     \begin{align}
     \label{eq:BL-tildeW1}  \widetilde{\mathsf W}_{1}^{2}(P ,Q ; \Lambda_1) & \leqslant  \,c {\sum_{0< \|m\| \leqslant t_0}\|m\|^{-2} \, | f_{P-Q}(m) | ^{2}}+c \, t_0^{-2},\\
\label{eq:BL-general}
 \widetilde{\mathsf C}^{2}_{w}(P ,Q ; \Lambda_1) & \leqslant  \, c{A_q^2} \,  {\sum_{0< \|m \| \leqslant t_0} \, {q(\|m\|)} \, w(\|m\|^{-1})^{2}| \, f_{P -Q }(m) | ^{2}}+c \, w(t_0^{-1})^{2}, \\
 \label{eq:BL-general1}
 \widetilde{\mathsf C}^{2}_{w}(P ,Q ; \Lambda_1) & \leqslant  \, c \, \dy{ \ln(t_0)} \,  {\sum_{0< \|m \| \leqslant t_0} w(\|m\|^{-1})^{2} \, | f_{P -Q }(m) | ^{2}}+c \, w(t_0^{-1})^{2}.
 \end{align}
}
\end{theorem}
{In case of $C_w$, \eqref{eq:BL-general} was proven in \cite[Proposition 7.1]{bobkov2020transport} (see also \cite{bobkov2024correction}) for probability measures supported in $[0,\pi]^d$ and \eqref{eq:BL-tildeW1} in \cite[Proposition 2]{bobkov2021simple}. Even if the necessary estimates to derive the above bounds are present in these papers, the bounds aren't stated in the form as above and so we sketch the details now. We will make more exact references to the results we borrow from \cite{bobkov2020transport} and \cite{bobkov2021simple}. \dy{Note that in \eqref{eq:BL-tildeW2}, $f_{P} = f_{P-Q}$ as $f_Q(m) = 0, m \in \Z^d$ for uniform probability measure on $\Lambda_1$.}

We shall prove \eqref{eq:BL-general}, \eqref{eq:BL-general1}, \eqref{eq:BL-tildeW1} and \eqref{eq:BL-tildeW2}, in that order. 
\begin{proof}
Let $H$ be a vector with finite second moment and bounded Fourier support (included in the {unit ball of $\R^d$}) and $\tilde{H}$ be its coordinate-wise projection onto $\Lambda_1$ i.e., each co-ordinate of $H$ is projected onto $(-\pi,\pi]$ via the map
$$ \R \ni x \mapsto x - 2\pi k \in (-\pi,\pi] \, \, \, \text{if} \, \, \, \pi(2k-1) < x  \leqslant \pi (2k+1), \, \, k \in \Z.$$
Let $ \tilde P ,\tilde Q $ the smoothed measures obtained by convoluting $ P ,Q $ with $ \widetilde{t_0^{-1} H}$ (\dy{the projection of $t_0^{-1}H$ onto $\Lambda_1$}), with the convolution considered in the Torus $\Lambda_1$.

We skip $\Lambda_1$ from $\tilde{C}_w$, $\tilde{W}^2_1$, $\tilde{W}^2_2$ for convenience. Estimating via \dy{the dual representation} that (see also \cite[(7.3)]{bobkov2020transport})
$$ \max\{\tC_{w}(P,\tilde P),\tC_{w}(Q,\tilde Q)\}  \leqslant {3 w\big( \mathbf{E}| \widetilde{t_0^{-1} H}| \big)  \leqslant 3 w\big( t_0^{-1}\mathbf{E}|H| \big)}$$
and using triangle inequality 
, we obtain that
\begin{align}
\label{e:triangleineqCw}
\tC_{w}(P ,Q )\leqslant  \tC_{w}(\tilde P ,\tilde Q )+ 6w(t_0^{-1}\mathbf{E} | H | ).
\end{align}
Now, \cite[(4)]{bobkov2024correction} yields that
\begin{align}
\label{e:CwtPTqbd}
\tC_{w}(\tilde P ,\tilde Q )\leqslant cA_q\sqrt{\sum_{m\neq 0}q(\|m\|)w(\|m\|^{-1} )^{2} | f_{ \tilde 	P }(m)-f_{ \tilde 	Q }(m) | ^{2}}.
\end{align}
{Note that we have used that $w(\|m\|^{-1}\pi\sqrt{d/2})  \leqslant \lceil \pi\sqrt{d/2} \rceil w(\|m\|^{-1})$ as $w$ is non-decreasing and sub-additive.} Using that $\widetilde{t_0^{-1} H}$ and $t_0^{-1}H$ have same Fourier-Stieltjes transform on integer vectors, we can derive that
\begin{align*}
 f_{ \tilde P }(m)  = f_{ P }(m)f_{\widetilde{t_0^{-1} H}}(m) = f_{ P }(m)f_{t_0^{-1}H}(m), \, \, m \in \Z^d.
\end{align*}
%
%
\begin{align}
\label{e:fouriertPtQ}
| f_{ \tilde 	P }(m)-f_{ \tilde 	Q }(m) | & \leqslant |f_{ P - Q }(m)| \times |f_{t_0^{-1}H}(m)| \leqslant   |f_{ P - Q }(m)| \, {\1[\|m\|  \leqslant t_0]},
\end{align}
\dy{where we have also used that the support of $f_H$ is in the unit ball.}{ Substituting this in \eqref{e:CwtPTqbd} and \eqref{e:triangleineqCw} completes the proof of \eqref{eq:BL-general}. When squaring both sides, we have used that $(a+b)^2  \leqslant 2a^2 + 2b^2$ and absorbed the $2$ into the constant $c$.} \\

{To prove \eqref{eq:BL-general1} (the analogue of \cite[(7)]{bobkov2024correction}), we observe that in \eqref{eq:BL-general}, it suffices to choose $q$ defined only on $[1,t_0]$. By non-decreasing property, $q(\|m\|)  \leqslant q(t_0)$ and then one can check that $A_q^2q(t_0)$ is minimized over $q$ when $q(2^k)$ is constant for $2^k  \leqslant t_0$. This yields that $A_q^2q(t_0) \leqslant  c \ln(t_0)$  and thereby giving us \eqref{eq:BL-general1}.}

{As for \eqref{eq:BL-tildeW1}, \cite[Lemma 1]{bobkov2021simple} gives that
$$ \tilde{W}_1(\tilde P ,\tilde Q )\leqslant c \sqrt{\sum_{m\neq 0} w(\|m\|^{-1} )^{2} | f_{ \tilde 	P }(m)-f_{ \tilde 	Q }(m) | ^{2}}.$$
Now using this in place of \eqref{e:CwtPTqbd} and following the rest of the derivation of \eqref{eq:BL-general}, we immediately obtain \eqref{eq:BL-tildeW1}.}

Now, we are left with the proof of \eqref{eq:BL-tildeW2}. First, we prove a lemma comparing $P,\tilde P$.
\begin{lemma}
\label{l:W2PtP}
        \begin{align*}
\tW_{2}(\tilde P ,P )\leqslant ct_0^{-1}.
\end{align*}
\end{lemma}
\begin{proof}
 {Suppose that $P$ is an atomic measure. Denote the probability distribution of $\widetilde{t_0^{-1} H}$ by $q_{t_0}$.}
  {Recall that for probability measures $\mu ,\mu ',\nu ,\nu '$ on, say, $\Lambda _{1}$, and a coefficient $t\in [0,1],$
 the contractivity principle yields
\begin{align*}
\tilde W_{2}^{2}(t \mu+(1-t )\mu ' ,t \nu +(1-t) \nu ')\leqslant t \tilde W_{2}^{2}(\mu ,\nu )+(1-t ) \tilde W_{2}^{2}(\mu ' ,\nu '),
\end{align*}
because one naturally builds a (non-necessarily optimal) coupling between $t\mu +(1-t)\mu '$ and $t\nu +(1-t)\nu '$ from optimal couplings on resp. $\mu ,\nu $ and $\mu' ,\nu  '.$
 Iterating this inequality yields}
\begin{align*}
\tilde W_{2}^2(\tilde P ,P ) &= \tilde W_{2}^{2}(\frac{ 1}{l}\sum_{i}\delta _{x_{i}}\star q_{t_0},\frac{1}{l}\sum_{i}\delta _{x_{i}})
 \leqslant \frac{ 1}{l}\sum_{i}\tilde W_{2}^{2}(\delta _{x_{i}}\star q_{t_0},\delta _{x_{i}})
\end{align*}
and
\begin{align*}
\tilde W_{2}^{2}(\delta _{x}\star q_{t_0},\delta _{x})= \int_{}\|y\|^{2}q_{t_0}(y)\md y = O(t_0^{-2})
\end{align*}
using the fact that $ H$ has a finite {second} moment. {Now we extend to arbitrary probability measures $P$ by an approximation argument. Observe that by separability of probability measures on a compact Polish space under the $W_2$ distance \cite[Corollary 8.11]{ambrosio2021lectures}, we have that a general probability measure $P$ can be approximated in $\tilde W_2$ by a sequence of atomic measures $P_l, l \geq 1$ i.e., $\tilde W_2(P,P_l) \to 0$ as $l \to \infty$. Further, weak convergence of $\tilde P_l$ to $\tilde P$ along with quadratic moment convergence (as $\Lambda_1$ is compact) gives that  $\tilde W_2(\tilde P,\tilde P_l) \to 0$ as $l \to \infty$ (see for example, \cite[Theorem 8.8]{ambrosio2021lectures}). Now the Lemma follows for arbitrary probability measure $P$ from these two approximations and triangle inequality.}
\end{proof}
 {Now, assume that {$Q$ is the uniform distribution on $\Lambda_1$} and $P$ has a density $\varphi$. Using equation (7) from \cite{bobkov2021simple} (see discussion below therein), we derive using \eqref{e:fouriertPtQ} and that $f_Q(m) = 0$ for all $m \in \Z^d$,}
\begin{align}
\tilde W^2_{2}(\tilde P , \tilde Q ) & \leqslant  c\sum_{m\neq 0}\frac{ 1}{\|m\|^{2}} | f_{ \tilde P - \tilde Q }(m) | ^{2}  \leqslant  c'\sum_{m\neq 0}\frac{ 1}{\|m\|^{2}} | f_P(m) | ^{2} \nonumber \\
& \leqslant c' \sum_{1\leqslant  \|m\| \leqslant t_0}\frac{ 1}{\|m\|^{2}} | f_{P}(m) | ^{2} \, + \, c' \sum_{\|m\| \geqslant t_0}\frac{ 1}{\|m\|^{2}} | f_{P}(m) | ^{2}, \nonumber \\
\label{e:W2tildebd} & \leqslant c'\sum_{1\leqslant  \|m\| \leqslant t_0} \frac{ 1}{\|m\|^{2}} | f_{P}(m) | ^{2} \, + \, c' t_0^{-2},
\end{align}
{where in the last step, we have used the boundedness of $|f_P(m)|$.}
If $ P $ does not have a density, {$P$ can be approximated in $ \tilde W_{2}^{2}$ by a sequence of measures $P'_{l}, l \geq 1$ with density $l^{-d}P(C)$ on each cube $C$ of the grid $l^{-1}\mathbb{Z} ^{d}\cap \Lambda _{1}$. For such measures $P'_l$, we have that
$$ \tilde W_{2}^{2}(P'_{l},P)\leqslant \sum_{C}\Leb(C)  { \rm diam}(C)^{2}\to 0 \, \, \textrm{ as } l\to \infty.$$
Further, as in the last part of the proof of Lemma \ref{l:W2PtP}, $\tilde W_{2}^{2}(\tilde P'_{l},\tilde P) \to 0$ as $l \to \infty$. Thus \eqref{e:W2tildebd} holds for arbitrary probability measures $P$ as well.} \\

Now combining \eqref{e:W2tildebd} and Lemma \ref{l:W2PtP} with triangle inequality as in \eqref{e:triangleineqCw}, we obtain \eqref{eq:BL-tildeW2}.
\end{proof}

\subsection{Proof of Theorems \ref{thm:Wbd-infinite} and \ref{thm:Wbd-infinite-HU}}
\label{sec:proof_Wbd-infinite}
{Recall that $ \mu $ is a stationary point process with unit intensity.} Set \dy{$\hat{\Leb}_{1} := \frac{ 1}{(2\pi )^{d}}\Leb \1_{\Lambda_{1}}$}, the uniform probability measure on $\Lambda_1$. The proof is based on Theorem  \ref{thm:bd-finite-sample}  with the rescaled sample on $ \Lambda _{1}$
\begin{align*}
\hat \mu_{n}= N^{-1}\mu (\1_{\Lambda _{n}}n^{1/d}\cdot )
\end{align*}
where $N= \mu (\Lambda _{n})$ and set $\hat{\mu}_n = \delta_0$ if $ N = 0$. Observe that $\hat{\mu}_n,\hat{\Leb}_1$ are probability measures \dy{on $\Lambda_1$}.  Recall also the (not-rescaled) renormalised sample
\begin{align*}
\tilde \mu _{n}= \frac{ n}{N}\mu\1_{\Lambda _{n}}.
\end{align*}
with mass $ n$ and again $\tilde \mu_n = n\delta_0$ if $ N = 0$. \dy{ The}  {\it scattering intensity }  of $\dy{\mu _{n} = \mu(\1_{\Lambda_n}\cdot)}$ is defined as
\begin{align}
\label{eq:scattering}
S_{n}(k)= \frac{ 1}{N} \left|
  \dy{\int_{\Lambda_n}e^{-i\,  k \cdot x}\mu(\md x)}
\right| ^{2}
\end{align}
and $ S_{n}\equiv 0$ if $N  = 0.$   Thus for $ m\in \mathbb{Z} ^{d}$
\begin{align}
\label{e:fmu_scat_intensity}
 | f_{ \hat  \mu _{n}}(m) | ^{2} = \frac{ 1}{N^{2}} \left| \dy{\int_{\Lambda_n}e^{-i \, n^{-1/d} \, m \cdot x}\mu(\md x)} \right| ^{2}= \frac{ 1}{N}S_{n}(mn^{-1/d}),
\end{align}
 { and set $f_{ \hat  \mu _{n}}(m)  = 1$  {if $N = 0$}.} \dy{The scattering intensity can be seen as an empirical estimate of the structure factor $S$. More precise convergence rates are given below and are useful for our proofs.}

 {Recall the definitions of   $S,\bn, \varepsilon(t)$ from \eqref{e:structfact},\eqref{eq:bn} and \eqref{eq:def-eps} respectively. To prove Theorems \ref{thm:Wbd-infinite}  {and \ref{thm:Wbd-infinite-HU}}, we need some estimates on variance and scattering intensity. The first lemma is a quantitative refinement of \cite[Proposition 2]{martin1980charge} and the second borrows a trick from \cite{hawat2023estimating}, probably well known in the physics literature}. It quantifies the  {rate of convergence of} the scattering intensity to the structure factor; see \RL{also \cite[Theorem 5.1]{coste2021order} for a more general weak convergence result}. 

\begin{lemma}\label{lm:var-bd}
{It always holds that
\begin{align*}
\var(N)\leqslant cn \bn .
\end{align*}
\dy{Let $\pi_d := (2\pi)^d$.} If $ \beta $ is integrable,
\begin{align}
\label{eq:bd-var-int}
\var(N)= n \pi_d \Big( \beta (\mathbb{R}^{d})+ 1 + O(\varepsilon (n^{-1/d}) \Big).
\end{align}
Furthermore, if  $\beta(A)\leqslant 0 $ for all $ A\subset \mathbb{R}^{ d}$,
\begin{align*}
\var(N)\geqslant n \pi_d \Big( 1 + \beta (\mathbb{R}^{d}) + \varepsilon (n^{-1/d}) \Big).
\end{align*}
}
\end{lemma}

\begin{lemma}
\label{lm:approx-sf}
 {Let $c_0 > 0$.} Let $ k= mn^{-1/d}$ for some $ m\in \mathbb{Z} ^{d}$ such that $ 0< \| m \| \leqslant t_0:= [c_{0}n^{1/d}].$  If $ \beta $ is integrable
\begin{align}
\label{eq:lm-4-1}
\mathbf{E}\left( \frac{\pi_d n}{N}S_{n}(k) \right) = S(k)+\frac{ \var(N)}{\pi_d n }+O(\varepsilon (n^{-1/d})).
\end{align}
and without  {any} assumptions on $ \beta $
\begin{align}
\label{eq:lm-4-2}
\mathbf{E} \left( \frac{ n}{N}S_{n}(k) \right) \leqslant  c{\bn }
\end{align}
\end{lemma}

Let us conclude the proofs of Theorems \ref{thm:Wbd-infinite}  {and \ref{thm:Wbd-infinite-HU}} before proving the lemmata. The proof uses Theorem  \ref{thm:bd-finite-sample} by rescaling $\tilde \mu_n$ and $\tilde \Leb_n$ to probability measures $\hat \mu _{n},\hat{\Leb}_{1}$ on $\Lambda_1$ and comparing the same to the costs on $\Lambda_n$. \dy{Note that for all $m \in \mathbb{Z}^d$, $f_{\hat{\Leb}_1}(m) = 0$ as it is the uniform probability measure on $\Lambda_1$ and so $f_{ \hat \mu _{n}-\hat{\Leb}_1}(m) = (f_{ \hat \mu _{n}}-f_{\hat{\Leb}_1})(m) = f_{ \hat \mu _{n}}(m)$.}

\begin{proof}[Proof of Theorem \ref{thm:Wbd-infinite}]
We again skip $\Lambda_n$ from $\tilde{W}_p, p \in [0,1] \cup \{2\}$ for convenience.  {Fix $t_0 = n^{1/d}$.} By \eqref{e:fmu_scat_intensity} and Lemma  \ref{lm:approx-sf},  $ \mathbf{E}  ( | f_{ \hat \mu _{n}}(m) | ^{2})\leqslant c \bn n^{-1}$  {for $\|m\| \leq t_0$}. So using Theorem  \ref{thm:bd-finite-sample} (in fact \eqref{eq:BL-tildeW1} and \eqref{eq:BL-tildeW2} \dy{with $f_{\hat{\Leb}_1}(m) = 0$}) yields for $ p= 1,2$
\begin{align*}
\mathbf{E}  \widetilde{  \mathsf W}^{2}_{p}(\hat \mu _{n},\hat{\Leb}_{1})
 \leqslant & \,
 {n^{-1}\sum_{0< \| m \| \leqslant t_0}\frac{c}{\| m \|^{2}}}\bn +cn^{-2/d}
 \\
\leqslant & \,
c
n^{-1}  \bn {\sum_{j= 1}^{n^{1/d}}\frac{ j^{d-1}}{j^{2}}}+cn^{-2/d}
\\
\leqslant & \, c\bn \begin{cases}
n^{-1} $  if $d= 1\\
n^{-1}{\ln(n)}$  if $d= 2\\
n^{-1}{(n^{1/d})^{d-2}}= n^{-2/d}$  if $d\geqslant 3
 \end{cases}
\end{align*}

Fix $ p\in (0,1)$. {We can derive three different bounds in this case and then choose the best possible. Firstly, two successive uses of H\"{o}lder's inequality and noting that $\hat{\mu}_n$ and $\hat{\Leb}_1$ are probability measures, gives the following bound}
\begin{equation}
\label{eq:WpHolderW1}
 \mathbf{E}   \widetilde{\mathsf W }_{p}^{2p}(\hat \mu _{n},\hat{\Leb}_{1}) \leqslant \big( \mathbf{E} \widetilde{\mathsf W }_{1}^{2}(\hat \mu _{n},\hat{\Leb}_{1}) \big)^{p}.
\end{equation}
{Secondly, using \eqref{eq:BL-general}, we derive}
\begin{align}
\mathbf{E}   \widetilde{\mathsf W }_{p}^{2p}(\hat \mu _{n},\hat{\Leb}_{1})
 \leqslant & \,
c n^{-1} \, \sum_{0< \| m \|  \leqslant n^{1/d}}\frac{c \, {q(\|m\|)}}{ \| m \| ^{2p}} \bn +n^{-2p/d} \nonumber \\
\label{eq:tWpmunbound} \leqslant & \, c
n^{-1} (\bn   + 1){\sum_{j= 1}^{n^{1/d}} \, {q(j)} \, j^{d-1-2p}}+n^{-2p/d}.
\end{align}
{Finally, using \eqref{eq:BL-general1} \dy{and that $q(t_0) \leq c\ln(n)$}, we derive}
\begin{equation}
\label{eq:tWpmunbound1}
\mathbf{E}   \widetilde{\mathsf W }_{p}^{2p}(\hat \mu _{n},\hat{\Leb}_{1})  \leqslant \, c
n^{-1} (\bn   + 1)\, \ln(n) \, {\sum_{j= 1}^{n^{1/d}} \, j^{d-1-2p}}+n^{-2p/d}.
\end{equation}
{Now using \eqref{eq:tWpmunbound} for $d = 1, p > 1/2$ with $q(x) = \ln(2x) \ln \ln(3x)^{\gamma}$ for $\gamma > 1$ (which is summable in this case), using \eqref{eq:tWpmunbound1} for \dy{$d =1, p \leq 1/2$} and \eqref{eq:WpHolderW1} in the remaining cases, we obtain that}
\begin{align*}
\mathbf{E} \widetilde{\mathsf W }_{p}^{2p}(\hat \mu _{n},\hat{\Leb}_{1})\leqslant & \, c  \begin{cases}
\bn \, n^{-1}$  if $d= 1,p>1/2 \\
{\bn \, n^{-1} \ln(n)^2}$  if ${d = 1, p = 1/2}, \\
{\bn \, n^{-2p} \ln(n)}$ if ${d = 1,p < 1/2}, \\
{ \bn^p \, n^{-p} \ln(n)^p}$  if ${d =2}, \\
\bn^p \, n^{-2p/d},$  if $d \geq 3 , \\
 \end{cases}
\end{align*}

Going back to large scale: For $ p>0,$
let $ \hat M_{n}$ be a  $\widetilde{\W }_{p}-$optimal transport plan between $ \hat \mu _{n}$ and $\hat \Leb_{1}$, and $ M_{n}(\md x,\md y):=n \hat  M_n(n^{-1/d}\md x,n^{-1/d}\md y)$ is a $ \W_{p}$ transport plan between $ \tilde  \mu _{n}$ and $ \tLeb_n$. Note that $M_n$ may not necessarily be optimal. Then we can derive
\begin{align}
\notag   \tW_{p}^{p} (\tilde \mu _{n},\tLeb_{n}) \, {\leqslant } \, & \int_{\Lambda _{n}\times \Lambda _{n}}d_{\Lambda _{n}}(x,y)^{p}M_{n}(\md x,\md y)\\
\notag= &\int_{\Lambda _{1}^{2}}n^{p/d}d_{\Lambda _{1}}(x,y)^{p}n \hat M_{n}(\md x,\md y)\\
\label{eq:scale}= &n^{1+p/d}\tW_{p}^{p}(\hat \mu _{n},\hat \Leb_{1}),
\end{align}
 {and when combined with the bounds for $\mathbf{E} \widetilde{\mathsf W }_{p}^{2p}(\hat \mu _{n},\hat{\Leb}_{1})$, this yields} the rates  \eqref{eq:macro-rate-std}-\eqref{eq:macro-rate-std-p}.
\end{proof}

\begin{proof}[Proof of Theorem \ref{thm:Wbd-infinite-HU}]

We shall again skip $\Lambda_n$ from $\tilde{W}_p, \, p \in  \{1,2\}$ for convenience. \\

\noindent (1). \, In the integrable HU case, the bound on $ \mathbf{E}( | f_{ \hat \mu _{n}} (m)| ^{2})$ can be improved and this forms the basis of our proof. {Set $k = mn^{-1/d}$ for $m \in \Z^d \setminus \{0\}$} {and $t_0 = c_0n^{1/d}$}. We have by {Lemma \ref{lm:var-bd} that $ n^{-1}\var(N)= O(\varepsilon (n^{-1/d}))$ and hence by Lemma \ref{lm:approx-sf} for $\|m\| \leq t_0$,}
\begin{align}
\label{eq:claim-eps}
\mathbf{E}(| f_{ \hat \mu_{n} }(m) | ^{2})=
\mathbf{E}\left( \frac{ 1}{N}S_{n}(k) \right) \leqslant  \frac{ 1}{n}S(k)+ n^{-1}O(\varepsilon (n^{-1/d})).
\end{align}
Also,the HU property and %
\begin{align*}
 | e^{-ik\cdot x}-1 | \leqslant c \min(1, | k\cdot x | ) \, \, \textrm{and} \, \, S(k)= &1+\int_{}e^{-ik\cdot x}\beta (\md x)=   \int_{} \,( \, e^{-ik\cdot x} - 1 \, )\beta (\md x)
 \end{align*}
\dy{ gives us that}
 \begin{align*}
 | S(k) | \leqslant  & \, c \, \int_{}\min(1, | k\cdot x | ) | \beta  | (\md x) \leqslant \, c \, \varepsilon (\| k \|  ),
\end{align*}
recalling  \eqref{eq:def-eps}.
It is useful at this stage to remark that $ \varepsilon $ is non-decreasing.
In particular, $ \varepsilon (n^{-1/d})\leqslant \varepsilon (\| k \| )$  {as $\|k\| = \| m \|n^{-1/d} \geqslant  n^{-1/d}$}. It yields
\begin{align*}\mathbf{E}(| f_{ \hat \mu_{n} }(m) | ^{2})\leqslant \, c \, n^{-1}\varepsilon ( \| k \| ).
\end{align*}
 {For $p=1,2$ using this monotonicity, \eqref{eq:BL-tildeW1} and \eqref{eq:BL-tildeW2}, we derive},
\begin{align*}
\mathbf{E} \widetilde{   \mathsf W}^{2}_{p}(\hat \mu _{n},\hat{\Leb}_1)
 \leqslant & {\frac{c}{n}\sum_{0< \| m \| \leqslant t_0} \frac{ \varepsilon (\|m\|n^{-1/d})}{\|m\|^{2}}}+cn^{-2/d}\\
\leqslant & c \, n^{-1 }{\int_{1}^{c_{0}n^{1/d}} \frac{\varepsilon (rn^{-1/d})}{r^{2}}r^{d-1}\md r} + n^{-2/d}\\
\leqslant &c \, n^{-1}{\int_{n^{-1/d}}^{c_{0}} { \varepsilon (r)} (n^{1/d})^{d-2}r^{d-3}\md r }+n^{-2/d}\\
\leqslant & c \, n^{-2/d}{\int_{n^{-1/d}}^{c_{0}}\varepsilon (r)r^{d-3}\md r}+n^{-2/d}
\end{align*}
which  yields, using  \eqref{eq:scale}, that the rate is indeed given by  \eqref{eq:rate-hu}, i.e.
\begin{align}
\label{eq:rate-hu1}
{\mathbf{E}\widetilde{   \mathsf W}_{2}^{2} (\tilde \mu _{n},\tLeb_{n})}\leqslant c \, n \left(
{1+\int_{n^{-1/d}}^{c_0}\varepsilon (r)r^{d-3}\md r}
\right) \dy{= \alpha_2(n)} ,
\end{align}
with a similar bound for $ {\mathbf{E}\widetilde{   \mathsf W}_1^{2} (\tilde \mu _{n},\tLeb_{n})}$. \\

\noindent (2). \, In $d = 2$, since
\begin{align*}
\int_{0}^{1}\min(1,r | x | )r^{d-3} \md r & =   |x|\1[|x| < 1]   + (\ln |x| +  1) \1[|x| > 1]  ,
\end{align*}
the assumed logarithmic integrability condition of $\beta$ and Fubini's theorem implies that
$$ \int_0^{1} \varepsilon(r)r^{d-3} \md r  =  \dy{\int |x|\1[|x| < 1] |\beta|(\md x)} + \int \1[|x| > 1] \int_{0}^{1}\min(1,r | x | )r^{d-3} \md r |\beta|(\md x) < \infty.$$
So, we have that $\alpha_2(n) \leqslant cn$ in \eqref{eq:rate-hu} where we have chosen $c_0 = 1$. Now the claim follows from \dy{item (1) of the theorem i.e., \eqref{eq:rate-hu1}}.
\end{proof}

\REMOVE{  \footnote{\RL{Ok to remove}}  \sout{\noindent (3). \,  Consider \eqref{eq:rate-hu} with $c_0 = 1$. We shall bound $\int_{n^{-1}}^{1} \varepsilon(r)r^{-2} \md r$ as follows. First observe that}
\begin{align*}
\int_{n^{-1}}^{1}\min(1,r | x | )r^{-2} \md r  & = |x| \ln(n) \, \1[ \, |x| \leq 1 \,] + (n-1) \, \1[ \, |x| \geq n]
  \\
  & \quad + \1[ \, 1 < |x| < n \,] \left[ \, |x| (\ln(n) - \ln(|x|) + 1)  -1\,  \right].
 \end{align*}
\dy{\sout{So integrating with $|\beta|(\md x)$, we derive that}}
\begin{align*}
|\int_{n^{-1}}^{1} \varepsilon(r)r^{-2} \md r | & \leq \ln(n)|\beta|(|x| \leq 1) + n|\beta|(|x| \geq n) + c\ln(n) \int  \1[ \, 1 < |x| < n \,] |x| \, |\beta|(\md x) \\
& \leq c \ln(n),
\end{align*}
\dy{\sout{where in the last inequality we have used $\int |x||\beta|(\md x) < \infty$ to deduce that $n|\beta|(|x| \geq n) \to 0$ and the finiteness of the last integral. Now again by Item (1) of the theorem, the claim follows.}}}

It remains to prove the lemmata at the beginning of this subsection. For that, define the functions
\begin{align*}
\delta _{n}(x) := & | \Lambda _{n}\cap (\Lambda _{n}+x) | \, \, \textrm{and} \, \,
\gamma _{n}(x) :=  | \Lambda _{n}\cap (\Lambda _{n}+x)^{c} | =  | \Lambda _{n} | -\delta _{n}(x).
\end{align*}

\begin{lemma}
\label{lm:gamma_n}
We have for all $x \in \R^d, n \in \mathbb{N}$,
\begin{align}
\label{eq:bd-gamma_n}
0<c\leqslant \frac{ \gamma _{n}(x) }{n\min(1,\|x\|n^{-1/d})}\leqslant C<\infty
\end{align}
and hence
\begin{align*}
 | \int_{}\gamma _{n}(x)\beta (\md x) | \leqslant   c \, n \, \varepsilon (n^{-1/d}).
\end{align*}
The reverse inequality holds (up to a constant) if $ \beta $ has a constant sign.
\end{lemma}

\begin{proof}  {We shall prove \eqref{eq:bd-gamma_n} with $\ell_{\infty}$ norm and by equivalence of norms in $\R^d$, the statement holds also for Euclidean norm. Set $ \|x\| = \max_{i} | x_{i} |$ and $x_+  := \max \{x,0\}$ for $x \in \R$.}
{We have the explicit expressions
\begin{align*}
\delta _{n}(x)= \prod_{i= 1}^{d}(2\pi n^{1/d}- | x_{i} | )_{+}, \, \, \text{and} \, \, \gamma _{n}(x)&
 = \sum_{j =  1}^{d}(2\pi n^{ 1/d})^{j- 1} | x_{j} |  \prod_{k = j +  1}^{d}(2\pi n^{ 1/d}- | x_{k} | )_+,
\end{align*}
where the latter uses the following telescopic formula: for real numbers $ a_{i},b_{i}$, with the convention $ b_{d +  1} = a_{0} =  1, $ it holds that
\begin{align*}  \prod_{i}a_{i}-  \prod_{i}b_{i} = &\sum_{j =   1}^{d}(a_{ 1}\dots a_{j}b_{j +  1 }\dots b_{d}-a_{ 1}\dots a_{j- 1 }b_{j}\dots b_{d}),
\end{align*}
}If $ \|x\|\geqslant 2\pi n^{ 1/d}$, $ \gamma  _{n}(x) = (2\pi )^{d}n$, and the result holds. We have the general upper bound if $ \|x\|<2\pi n^{ 1/d}$ as then $ \gamma _{n}(x)\leqslant {d}\|x\|(2\pi n)^{\frac{ d- 1}{d}}\leqslant n$. We need to prove a lower bound in this case.

{Since $ \|x\| < 2\pi n^{ 1/d}$, there is} $1 \leqslant i_{0} \leqslant d$ such that
\begin{align*}
|x_{i_{0}}|  = \|x\| <2\pi n^{ 1/d}.
\end{align*}
Now if $ \|x\|>\pi n^{ 1/d}$ then $| \delta _{n}(x)| < | \Lambda _{n} | /2$, and $ \gamma _{n}(x)\geqslant cn,$ hence the lower bound follows in this case.
Else $ \|x\|< \pi n^{ 1/d}$ and then $ 2\pi n^{ 1/d}- | x_{k} | \geqslant  \pi n^{ 1/d}$ for all $k$, and so one can lower bound by the single term $ j = i_{0}:$
\begin{align*}\gamma _{n}(x)\geqslant cn^{\frac{ d- 1}{d}}|x_{i_{0}}| = c\|x\|n^{ 1- 1/d}.
\end{align*}
{Using the upper bound for $\gamma_n$ and the definition of $\varepsilon$, one can immediately bound $\int_{}\gamma _{n}(x)|\beta|(\md x)$ and in case $\beta$ has constant sign, we derive the lower bound using the lower bound for $\gamma_n$.}
\end{proof}

\begin{proof}[Proof of Lemma  \ref{lm:approx-sf}]
\dy{Recall that $\pi_d := (2\pi)^d$, the volume of the box $\Lambda_1$.} For $ k= mn^{-1/d}$
\begin{align}
\label{eq:approx-SI}
\mathbf{E}\left[
 \frac{ 1}{N}S_{ {n}}(k)
\right]=&\mathbf{E} \frac{ 1}{N^{2}}\left|
\sum_{x\in \mu  _{n}}e^{-ik\cdot x}
\right|^{2}
\leqslant \dy{\frac{4}{\pi_d^2}}\mathbf{E}\left(
\frac{ 1}{n^{2}}\left|
\sum_{x {\in \mu_n}}e^{-ik\cdot x}
\right|^{2}
\right)+\mathbf{E}\left(
\frac{ 1}{N^{2}}N^{2}\mathbf{1}_{\{N< \dy{\pi_d n}/2\}}
\right).
\end{align}
For the second term, Bienaym\'e-Chebyshev inequality yields
\begin{align*}
\mathbf{P}(N < \dy{\pi_d n}/2)\leqslant \mathbf{P}(| N-\mathbf{E}(N) | \geq  \dy{\pi_d n}/2)\leqslant \frac{4\text{\rm Var}(N)}{\pi_d^2n^{2}}
\end{align*}

The first term of  \eqref{eq:approx-SI} can be approximated by the structure factor: for $ k= mn^{-1/d}$,
\begin{align*}  {
\int_{\Lambda _{n}}e^{-ik \cdot x}\md x
}  = 0,
\end{align*}
and hence following the trick from \cite{hawat2023estimating} via Campbell-Mecke-Little formula, we derive that
\begin{align*}
S(k)   = &1 +\int_{}e^{-ik \cdot x}\beta (\md x)\\
  = &1 + \int_{}e^{-ik \cdot x}\frac{ 1}{ | \Lambda _{n} | }\int_{}\1_{\Lambda _{n}}(y)\md y\beta (\md x)\\
= &1 + \dy{\frac{1}{\pi_d n}} \, \int_{}e^{-ik \cdot x} \, \int{\1_{\Lambda _{n}}(y)}\1_{\Lambda _{n}}(y+x)\md y\beta (\md x)+ \dy{\frac{1}{\pi_d n}}\int \int_{}e^{-ik \cdot x}\1_{\Lambda _{n}}(y)\1_{(\Lambda _{n}-x)^{c}}(y)\md y\beta (\md x)\\
= &1 + \dy{\frac{1}{\pi_d n}} \,\int  \int_{}e^{-ik \cdot (x+y)}\1_{\Lambda _{n}}(x+y)e^{-ik \cdot y}\1_{\Lambda _{n}}(y)\md y\beta (\md x)+
\dy{\frac{1}{\pi_d n}}\int_{}e^{-ik \cdot x}\gamma _{n}(-x)\beta (\md x)
\\
= &1 + \dy{\frac{1}{\pi_d n}} \, \mathbf{E}\dy{\int_{\Lambda_n^2}}\1[x\neq y]e^{-ik \cdot (x - y)}\mu (\md x)\mu (\md y) +\underbrace{\left | {
\int_{\Lambda _{n}}e^{-ik \cdot x}\md x
} \right | ^{2}}_{ = 0}+O\left(
\varepsilon (n^{-1/d})
\right)\\
= &  \dy{\frac{1}{\pi_d n}\mathbf{E}
\sum_{x {\in \mu_n}}e^{-ik \cdot (x-x)} + \dy{\frac{1}{\pi_d n}}\mathbf{E}
\sum_{x \neq y {\in \mu_n}}e^{-ik \cdot (x-y)}
+O\left(
\varepsilon (n^{-1/d})
\right) } \\
= &  \dy{\frac{1}{\pi_d n}}\mathbf{E}
\left| \sum_{x {\in \mu_n}}e^{-ik \cdot x} \right|^2 
+O\left(
\varepsilon (n^{-1/d})
\right)
\end{align*}
\dy{where we have used \eqref{eq:bd-gamma_n} in the fifth equality and \eqref{eq:def-beta} in the penultimate equality.} {In the fifth equality above, we have used that $\gamma_n$ and $\beta$ are even functions; the former by definition and the latter due to stationarity of $\mu$.}
Plugging back into  \eqref{eq:approx-SI} yields   \eqref{eq:lm-4-1}.

If instead nothing is assumed on $ \beta $ we can go backwards in the above computation and treat the first term of \eqref{eq:approx-SI} with
\begin{align}
\mathbf{E}  \int_{\Lambda_n^2}\1[x\neq y]e^{-ik \cdot (x - y)}\mu (\md x)\mu (\md y)  = & \int_{}e^{-i k \cdot (x+y)}\1_{\Lambda _{n}}(y)e^{-ik \cdot y}\1_{\Lambda _{n}}\dy{(y+x)}\md y\beta (\md x)  \nonumber \\
\leqslant & | \Lambda _{n} | \int_{}\1_{\Lambda _{n} + \Lambda _{n}}(x) | \beta |  (\md x). \label{eq:betalambdan_bd}
\end{align}
Hence in the non-integrable case, \dy{recalling $b_n$ from \eqref{eq:bn} and together} with Lemma  \ref{lm:var-bd}
\begin{align*}
\mathbf{E}(N^{-1}S_{n}(k))\leqslant \frac{\pi_d n}{\pi_dn^{2}} \, {|\beta|} (\Lambda _{n} + \Lambda _{n})+\frac{ \var(N)}{\pi_d^2 n^{2}}\leqslant  c \, \frac{\bn}{n}
\end{align*}
which ultimately gives \eqref{eq:lm-4-2}.
\end{proof}

\begin{proof}[Proof of Lemma  \ref{lm:var-bd}]
Using \eqref{eq:def-beta}, we can derive
\begin{align*}
\var(N)= &\mathbf{E}(N^{2})-(\mathbf{E}(N))^{2}\\
= &   | \Lambda _{n} |  +\mathbf{E}(\sum_{x\neq y\in \mu _{n}}1)- | \Lambda _{n} | ^{2}\\
= &  | \Lambda _{n} | +\int_{\Lambda _{n}^{2}}\md x(\md y+\beta (d(y-x)))- | \Lambda _{n} | ^{2}\\
= &   | \Lambda _{n} |  +\int \1[x\in \Lambda _{n},x + z\in \Lambda _{n}]\md x\beta (\md z).
\end{align*}
In the non-integrable case, \dy{arguing as in \eqref{eq:betalambdan_bd},} we obtain
\begin{align*}
\var(N)\leqslant cn \bn .
\end{align*}
In the integrable case, we can write as
\begin{align}\notag
\var(N)= & | \Lambda _{n} | +\int_{\Lambda _{n}\times \mathbb{R}^{d}}\md x\beta (\md z)-\int\gamma _{n}(z)\beta (\md z)\\
\label{eq:var-gamma-beta}= &\pi_d n(1+\beta (\mathbb{R}^{d}))-\int_{}\gamma _{n}(z)\beta (\md z),
\end{align}
and thereby combined with Lemma \ref{lm:gamma_n}, it yields  \eqref{eq:bd-var-int}. \\

If furthermore $ \beta \leqslant 0$, the second part of Lemma \ref{lm:gamma_n} yields
\begin{align*}
\var(N)\geqslant \pi_d n(1+\beta (\mathbb{R}^{d})+\varepsilon (n^{-1/d})).
\end{align*}
\end{proof}
  \subsection{Local to global allocation}
 \label{sec:global}

The following proposition formalizes and generalizes one of the steps in the proof of \cite[Theorem 1(iii)]{holroyd2009poisson} and is very important to deduce properties of global matchings via properties of matchings in finite windows.
Recall that for a measure $\mu$, $ \tilde  \mu _{n}= \frac{n}{N} \mu\1_{\Lambda _{n}}$ (with $\tmu_n = n\delta_0$ for $N=0$) and $\tLeb_n := \frac{1}{(2\pi)^d} \Leb \1_{\Lambda _{n}}$ {are} the renormalised samples with mass $n$ on $ \Lambda _{n}.$
\begin{proposition}
\label{p:matching_limit}
 \RL{Let $ \mu ,\nu $ be independent {simple point processes} with integrable RPCM and $w : [0,\infty) \to [0,\infty)$ be a {continuous increasing} weight function with $ w(0) = 0$.
Assume $\E\widetilde {\C}_{w}(\tmu _{n}, \tilde \nu _{n};\Lambda_n) \leqslant cn$   for all $n \geqslant  1$. Then, there exists  an invariant coupling $M$ of $\mu$  and $\nu$ {that is a.s. a matching and} whose typical distance $ X$ satisfies $ \mathbf{E}w(X)<\infty$.}
%
\end{proposition}
\dy{ While the existence of a coupling of $\mu,\nu$ proceeds by a compactness argument, the fact that the coupling is a matching is an infinite version of the Birkhoff-von Neumann theorem (\cite[Exercise 40]{santambrogio2015optimal}) which guarantees one-to-one correspondance in the case where $\mu,\nu$ are finite simple point processes with same number of points. For \RL{the infinite case}, we exploit a recent result of  \cite{erbar2023optimal}. However, a more explicit (but tedious) construction appears in a previous version of this work on arXiv.}

\REMOVE{
\obsolete{
This result is an infinite version of the well known Birkhoff-von Neumann theorem (\cite[Exercise 40]{santambrogio2015optimal}), which states that the optimal transport cost for.
\begin{proposition}\label{prop:birkhoff}
Let $ n\geqslant  1$, $ \mu = \sum_{i = 1}^{{n}}\delta _{x_{i}} ,\nu  = \sum_{i =  1}^{n}\delta _{y_{i}}$, where the $ x_{i},y_{i}$ are {distinct} points of $ \mathbb{R}^{d}$, and let $ w$ a non-decreasing cost function. Then the optimal coupling $ M$ is realised by a transport map.
\end{proposition}
The first part  of the proof of Proposition \ref{p:matching_limit} is a weak compactness argument, the second part proceeds by approximating $\tLeb_{n}$ by a re-scaled and re-normalized lattice. We exploit a recent result of  \cite{erbar2023optimal} for the second step, an explicit construction appears in a previous version of this work on arXiv.}}

Before proceeding with this proof, we need to recall some results on vague convergence from \cite[Section 11]{DvJ08}. \dy{For a measure $m$ on $\R^d$ and a measurable function $f : \R^d \to \R$, we denote $m(f) := \int f \, \md m$ whenever the integral is well-defined.} Given measures  $ m,m_{n},n\geqslant  1$ on $ \mathbb{R}^{d}$, say that $ m_{n}$ converges vaguely to $ m$, denoted by $ m_{n} \xrightarrow[]{v}m\;$, if $ m_{n}(f)\to m(f)$ for all  $f \in C_c(\R^d)$, the space of compactly supported (bounded) continuous functions on $ \mathbb{R}^{d}$.  \dy{For random measures $ M_{ n},M$, say $ M_{ n}$ converges to $ M$ vaguely in law if $M_{ n}(f) \convlaw M(f)$ for compactly supported (bounded) continuous $f$, where $\convlaw$ denotes convergence in law.} Recall from   \cite[Proposition 11.1.VI]{DvJ08}  that if a sequence of random stationary measures $ (M_{n})_{ n}$ satisfies
\begin{equation}
\label{e:tight_vague}
\sup_{n}\mathbf{E}M_{n}(A) < \infty
\end{equation}
for all  bounded $A$, then it is tight and hence there exists a subsequence that converges vaguely {in law}.

\begin{proof}[Proof of Proposition \ref{p:matching_limit}]

 {The first two steps are actually valid for any stationary random non-negative measure with finite intensity while the third step exploits the assumption of simple point processes, we believe that this can be avoided.} \\

 \underline { First step: Existence of \dy{  an invariant} infinite coupling $ M$:}  \\


\dy{Let $\epsilon > 0$. By definition \eqref{e:toric_cost}, for every realization} {there is a transport plan $ m_{n}$ from $ \tilde \mu_{n}$ to $ \tilde \nu _{n}$ such that}
$$\int _{\Lambda _{n}^{2}} w(d_{n}(x,y)) m_n(\md (x,y)) \leq \widetilde {\C}_{w}(\tmu _{n}, \tilde \nu _{n};\Lambda_n) + \epsilon n.$$
\dy{Thus by assumption $\E \widetilde {\C}_{w}(\tmu _{n}, \tilde \nu _{n};\Lambda_n) \leq cn$, we have that} 
\begin{equation}
\label{e:dnmn1}
\E \int _{\Lambda _{n}^{2}} w(d_{n}(x,y)) m_n(\md (x,y)) \leqslant cn,
\end{equation}
$m_n(\cdot,\R^d) = \tilde \nu _n$, $m_n(\R^d,\cdot) =  \tilde \mu_n$ where $d_n$ is the toroidal distance  {on $\Lambda_n$} \dy{and we have denoted the new constant as well by $c$, with an abuse of notation}. Suppose we tile $\R^d$ with randomly shifted copies of $\Lambda_n$ and place $ \tilde\mu_n,\tilde \nu _{n}$ in each of them
and \dy{match by $ m_{n}$ in each copy}. We call this transport map $M_n$.
Formally,
\begin{align}
\label{eq:stationarized}
M_n(A , B) = \sum_{z \in \Z^d} m_n\big((A-U_{n} + {2\pi}n^{1/d}z) , (B-U_{n} + {2\pi}n^{1/d}z) \big),\end{align}
for $A,B \subset \R^d$ and where $U_{n}$ is uniform in $\Lambda_n$ and independent of $\mu$ and $ \nu $. Observe that \dy{the translation-invariance of $\sum_{z \in \Z^d}\delta_{ {2\pi}n^{1/d}z - U_{n}}$ shows that $M_n$'s are  invariant or equivariant}, i.e.
\begin{align}
\label{eq:trans-inv-Mn}  {M_{n}(A + x,B + x)\equlaw M_{n}(A,B)},
\end{align}
for all $ A,B \subset \R^d$ and $x \in \R^d$. Note that $ M_{n}(A,B)= 0$ if $ A,B$ are contained in different shifted copies. \\

Note that $\E M_n(A,B) \leqslant \E M_n(A,\R^d)$ and $\E M_n(A,\R^d)$ is translation invariant in $A$. Thus $\E M_n(A,\R^d) = c_{n}|A|$ for some $c_{n} \in [0,\infty]$. For $A = \Lambda_n$, we have that  $\E M_n(A,\R^d) = \E \tmu_n(\Lambda_n) = n$ and so we deduce that $c_{ n} = 1$. Similarly, we can deduce that $\E M_n(\R^d,B) = |B|$. Hence, we have that $\mathbf{E} M_{n}(A,B)\leqslant \min(| A | , | B | )$. \\

In particular, $ \mathbf{E}M_{n}(K)<\infty $ for bounded $K \subset \R^d \times \R^d$. Thus, \dy{appealing to \eqref{e:tight_vague},} the sequence $M_n$ is tight in vague topology  and has a subsequential distributional limit under vague topology, say $M_{n'} \convlaw M$. \dy{For a compactly supported continuous function $f : \R^d \times \R^d \to \R_+$ and $x \in \R^d$, define $f_x(\cdot,\cdot) := f(x+\cdot,x+\cdot)$, 
also a bounded continuous function. By invariance of $M_n$'s, we have that $M_n(f_x) \equlaw M_n(f)$ and hence the subsequential convergence argued above gives that $M_n(f_x) \to M(f)$ and $M_n(f) \to M(f)$ as $n \to \infty$. Thus $M(f_x) \equlaw M(f)$, which gives that $M$ is invariant.} \\

\underline{{Second Step: $M$ has finite typical cost.}}  \\
\dy{Recalling \eqref{def:typical-dist} and \eqref{e:typical_cost_w}, to show that the typical transport distance under \RL{$ M$} has finite cost means that
\begin{equation}
\label{e:typ_M_cost}
\E \int \int w(\|x-y\|) \1[x \in \Lambda_1] M(\md (x,y)) < \infty.
\end{equation}
}
Denote $ \tilde d_{n}(x,y)$ to be the distance which is $ 0$ if $ x,y$ are in different {randomly shifted copies of $\Lambda_n$}, and the toric distance of the copy otherwise i.e.,
$$ \dy{ \tilde d_{n}(x,y) := \sum_{z \in \Z^d} d_n(x-U_n-{2\pi}n^{1/d}z,y-U_n-{2\pi}n^{1/d}z)}.$$
So $\E[\int \int w(\tilde d_{n}(x,y))\1[x \in A] M_n(\md (x,y))]$ is translation invariant in $A$ and hence its intensity is either infinity or of the form $ \textrm{const}\cdot \Leb$, and for $A = \Lambda_n$,  by \eqref{e:dnmn1}
$$ \E \int \int w(\tilde d_{n}(x,y))\1[x \in \Lambda_n] M_n(\md (x,y)) = \E \int _{\Lambda _{n}^{2}} w({d_{n}(x,y)}) m_n(\md (x,y)) \leqslant c \Leb(\Lambda_n).$$
Thus we obtain that, there exists a constant $c$ such that for all $A$
\begin{equation}
\label{e:WpMnA1}
\E \int \int w(\tilde d_{n}(x,y))\1[x \in A] M_n(\md (x,y)) \leqslant c \Leb(A).
\end{equation}

\dy{Now let $A = \Lambda_1$ and $B$ be a bounded set.} Let $ n_{0}$ be such that $ \sup_{x\in \Lambda_1,y\in B}\|x-y\|<n^{1/d}/2$ {for all $n \geqslant  n_0$}. Hence for $x \in \Lambda_1,y\in B$, either $ x,y$ are in different shifted copies of $\Lambda_n$ and $ M_{n}(\md x,\md y)= 0$ (and $\tilde d_{n}(x,y) = 0$), or they are in the same copy of $\Lambda_n$, but at distance $ < \frac{n^{1/d}}{2}$, hence $\tilde d_{n}(x,y)= \|x-y\|.$ In any case,
\begin{align*}
w(\|x-y\|)M_{n}(\md x,\md y)= w(\tilde d_{n}(x,y)) M_{n}(\md x,\md y).
\end{align*}

{For $ n\geqslant n_{0}$, from \eqref{e:WpMnA1}, it follows that}
\begin{align*}
\mathbf{E} \int \int w(\|x-y\|)\1[x \in \Lambda_1, y \in B] M_n(\md (x,y)) & = \E \int \int w(\tilde d_{n}(x,y))\1[x \in \Lambda_1, y \in B] M_n(\md (x,y)) \\
& \leqslant c \Leb(\Lambda_1) < \infty.
\end{align*}
\dy{We have shown in the First Step that $M$ is the subsequential distributional limit under vague topology of $M_n$. For convenience, we now assume that $M_n \convlaw M$ in vague topology.} \RL{For any continuous function $0\leqslant  f(x,y)\leqslant \1[x\in \Lambda_1,y\in B]$, the vague convergence in law yields that, with $ \tilde w(x,y) = w(\|x-y\|)f(x,y)$, 
$$X_{ n}: =  \int \int \tilde{w}(x,y) \, M_n(\md (x,y)) \convlaw X: =  \int \int \tilde{w}(x,y) \, M(\md (x, y)).$$
\RL{The upper bound then is a general feature of  convergence in law, through  a truncation argument for non-negative variables: for $ T>0$ fixed, $ \mathbf E X_{ n}\wedge T\to \mathbf E X \wedge T$, and $ \mathbf E X_{ n}\wedge T\leqslant c\Leb(A)$, so that with monotone convergence, we derive
$$\mathbf E  \int \int \tilde{w}(x,y) \, M(\md (x, y)) \leqslant c\Leb(\Lambda_1) < \infty.$$ 
}
 Then letting $ f(x,y)$ increase pointwise to $\1[x\in \Lambda_1]$ and using the monotone convergence theorem again  gives that}
$$ \E \int \int w(\|x-y\|) \1[x \in \Lambda_1 ] M(\md (x,y))  \leqslant c \Leb(\Lambda_1) < \infty$$ 
  as  {needed to verify finite typical cost \dy{in \eqref{e:typ_M_cost}}.} \\

 \underline {Third Step: $ M$ is a coupling of independent copies of $\mu$ and $\nu$}  \\
We will now show that $M$ is a transport plan between {independent copies of}  $\mu$ and $\nu $ defined  on a suitable probability space. 

Consider {\dy{non-negative} test functions $f,g \in C_c(\R^d)$} supported by resp. $ K,L$ compact convex sets. \RL{For $u \in \R^d$, let $\theta (K,u)$ be $ \infty $ if $ K-u $  touches one of the boundaries $\partial (z2\pi n^{ 1/d} + \Lambda _{ n}),z\in \mathbb{Z} ^{ d} $, and otherwise it is the element $ z\in \mathbb{Z} ^{ d}$ such that  $   K - u\subset (2\pi n^{ 1/d}z +  \Lambda _{ n})$.} Hence, under the event  $ \Omega_{n}: = \{\theta (K,U_{ n}) \neq  \infty \}$
%
%
\begin{align*}
(M_{n}(f,\mathbb{R}^{d}),M_{n}(\mathbb{R}^{d} ,g)) &= (\tilde \mu _{n}( \tilde f), \tilde \nu_{n}(\tilde g))
\end{align*}
where $ \tilde f(x) = f(x +  U_{ n} +  \theta (K,U_{ n})), \tilde g(x) = g(x + U_{ n} +  \theta (L,U_{ n}))$.
%
\dy{Recall that $\tilde \mu_n = \frac{n}{\mu(\Lambda_n)}\mu_n$ and similarly $\tilde \nu_n$.} {So, the random variables $\tilde \mu _{n}( \tilde f)$ and $\tilde \nu_{n}(\tilde g)$}  are independent because $\mu (\Lambda _{n}),\nu (\Lambda _{n})$ are independent and ${\mu_n}(\tilde f), {\nu_n}(\tilde g)$ are independent. Only the second point requires a proof:
 \dy{
\begin{align*}
\mathbf{E}e^{-\mu_n(\tilde f)-\nu_n(\tilde g)} &=  \mathbf{E}\bigl( e^{-\mu_n(\tilde f)-\nu_n(\tilde g)} \; \mid \;U_{n}) \bigr)\\
&=   {\mathbf{E}\bigl( \mathbf{E}( e^{-\mu_n(\tilde f)} \; \mid \; U_n) \, \mathbf{E}(e^{-\nu_n(\tilde g)} \;|\;U_{n}) \bigr)} \\
 & = \mathbf{E}(e^{-\mu_n(\tilde f)}) \, \mathbf{E}(e^{-\nu_n(\tilde g)})
\end{align*}
 where we have used that \dy{$\mu_n ( \tilde f),\nu_n ( \tilde g)$} are independent conditionally on $U_n$ in the second equality and in the third equality, we use stationarity and independence with $U_n$ as well as independence between $ \mu $ and $ \nu.$}
\REMOVE{for {Borel} sets $ I,J\subset \mathbb{R},$
\begin{align*}
\mathbf{P}(\mu_n(\tilde f)\in I,\nu_n(\tilde g)\in J) &=  \mathbf{E}\bigl( \mathbf{P}(\mu_n ( \tilde f)\in I, \nu_n ( \tilde g)\in J\; \mid \;U_{n}) \bigr)\\
&=   {\mathbf{E}\bigl( \mathbf{P}(\mu_n (  \tilde f)\in I \; \mid \; U_n) \, \mathbf{P}( \nu_n ( \tilde g)\in J\;|\;U_{n}) \bigr)} \\
 & = \mathbf{P}(\mu_n (\dy{\tilde f})\in I) \, \mathbf{P}(\nu_n (\dy{\tilde g})\in J)
\end{align*}
 where we have used that \dy{$\mu_n ( \tilde f),\nu_n ( \tilde g)$} are independent conditionally on $U_n$ in the second equality and in the third equality, we use stationarity and independence with $U_n$ as well as independence between $ \mu $ and $ \nu.$}
%
{Note that $\tilde \mu_n,\tilde \nu_n$ converge vaguely in law to $\mu,\nu$ respectively.} \dy{Relying on this and the independence of $\tilde \mu _{n}( \tilde f)$ and $\tilde \nu_{n}(\tilde g)$ shown above, we can derive that 
%
\begin{align*}
\mathbf{E}e^{-\mu (f)} \mathbf{E}e^{-\nu (g)}  = & \mathbf{E}e^{-\mu (\tilde f)} \mathbf{E}e^{-\nu (\tilde g)}  \\
 = & \lim_{n}\mathbf{E}e^{-\tilde \mu_n (\tilde f)} \, \mathbf{E}e^{-\tilde \nu_n(\tilde g)}\\
  = & \lim_{n}\mathbf{E}e^{-\tilde \mu_n (\tilde f) -\tilde \nu_n(\tilde g)} \\
 = & \lim_{n}\mathbf{E}\Big(e^{-\tilde \mu_n (\tilde f) -\tilde \nu_n(\tilde g)}\1_{\Omega _{n}}\Big) + O\Big(\mathbf{P}(\Omega _{n}^{c})\Big) \\
 = & \lim_{n}\mathbf{E}\Big(e^{-M_{n}(f,\mathbb{R}^{d})- M_{n}(\mathbb{R}^{d},g)}\1_{\Omega _{n}}\Big) + O(\mathbf{P}\Big(\Omega _{n}^{c})\Big) \\
= & \mathbf{E}e^{-M(f,\mathbb{R}^d)- M(\mathbb{R}^d ,g)}
\end{align*}
using} $ \mathbf{P}(\Omega _{n})\to 1$ {and without loss of generality, assuming $M_n$ converge vaguely to $M$ in law.} \dy{Thus we have that 
$$ \mathbf{E}e^{-(M(f,\mathbb{R}^d) + M(\mathbb{R}^d ,g))} = \mathbf{E}e^{-\mu(f)} \, \mathbf{E}e^{-\nu(g)} =  \mathbf{E}e^{-M(f,\mathbb{R}^d)} \, \mathbf{E}e^{-M(\mathbb{R}^d ,g)}.$$
} \dy{Now given compact sets $K,L$ and $s,t \geq 0$, there exist sequence of bounded continuous functions $f_n,g_n$ such that $f_n \downarrow t\1_K$ and $g_n \downarrow s\1_L$ respectively. By monotonic approximation, the above identity yields that  
$$ \mathbf{E}e^{- tM(K,\mathbb{R}^d) - sM(\mathbb{R}^d ,L)} =  \mathbf{E}e^{-tM(K,\mathbb{R}^d)} \, \mathbf{E}e^{-sM(\mathbb{R}^d ,L)},$$
thereby showing the independence of $(M(K,\mathbb{R}^d),M(\mathbb{R}^d ,L))$ for all compact sets $K,L$. This guarantees the independence of $(M(\cdot,\mathbb{R}^d),M(\mathbb{R}^d ,\cdot))$ via \cite[Corollary 1.3.15]{baccelli2020random} {and that $M(\cdot,\mathbb{R}^d),M(\mathbb{R}^d ,\cdot)$ are simple point processes}.} \\

\underline {\RL{Fourth Step: if $ \mu ,\nu $ are simple independent point processes, $ M$ can be chosen to be a.s. a matching}}  \\

\RL{This step is actually similar to the proof of Proposition 2.11 in  \cite{erbar2023optimal}. 
To be  precise, they define a distance   $   \mathsf C(  \mathscr  L_{ \mu },  \mathscr  L_{ \nu })$ between the laws $  \mathscr  L_{ \mu },  \mathscr  L_{ \nu }$ of resp. $ \mu ,\nu ,$ which satisfies with their Proposition 2.4
\begin{align*}
  \mathsf C(  \mathscr  L_{ \mu },  \mathscr  L_{ \nu })\leqslant \inf_{ M}\mathbf E \int_{\Lambda _{ 1} \times \mathbb{R}^{ d}}w(x-y) M(dx,dy)
\end{align*}
where the infimum is over \dy{  invariant} couplings $ M$ between independent versions $ \mu ',\nu '$ of $ \mu $ and $ \nu $ (note that independence is not required in their result). The RHS in fact corresponds with Fubini's theorem  to the definition of the typical distance from  \eqref{def:typical-dist} for the equivariant matching $ M$ built at   first step above \RL{(the invariance     comes from  \eqref{eq:trans-inv-Mn} which passes to the limit when $ M_{ n}\to M$)} : by definition of the typical distance,
\begin{align*}
\mathbf E w(X) = \int_{\mathbb{R}_{  + }}\mathbf E \int_{}{\bf 1}\{ w(\|x-y\|)\geqslant r \}M(dx,dy) dr= \mathbf E\int_{} \underbrace{\int_{}{\bf 1}\{w (\|x-y\|)\geqslant r \}dr}_{ = w(\|x-y\|)}M(dx,dy).\end{align*}}\RL{
Then  \cite[Proposition 2.11]{erbar2023optimal} allows to conclude directly that there is an equivariant matching between two versions $ \mu ',\nu '$ of $ \mu ,\nu $ which typical distance $ X_{ \text{\rm{match}}}$ satisfies $ \mathbf E w(X_{ \text{\rm{match}}})<\infty .$
}
\end{proof}

\begin{remark}
In the result of  \cite{erbar2023optimal}, it is not assumed that the point processes are simple. In fact, if they are not simple, there is not always a matching satisfying the conclusion of Proposition \ref{p:matching_limit}. It does not contradict \cite[Proposition 2.11]{erbar2023optimal} because a more general definition of a matching is used. \dy{Specifically, they define matching as a point process in the product space but not necessarily induced by an actual bijection. However, we define matchings via bijections and hence we need the assumption of simplicity for point processes.}
\end{remark}

\subsection{Proofs of Theorems \ref{thm:std} and \ref{thm:hu} }
\label{sec:proofs_thms12}

\dy{In both the proofs below, we first justify the case of independent simple point processes $\mu,\nu$ with integrable RPCM and then address the case of shifted lattices. For the latter, \RL{we use Example  \ref{ex:iidpertlattice}, a perturbed lattice that has an integrable RPCM, and allows to bridge the results between the two categories.} However, it is possible with some more effort to use \cite{erbar2023optimal} as in the last part of proof of Proposition \ref{p:matching_limit} to give a more direct argument.}

\begin{proof}[Proof of Theorem \ref{thm:std}]
Consider first the case $d \geq 3$ with the weight function $w(x) = x^2$. By Theorem \ref{thm:Wbd-infinite}, we have that there exists a finite constant $c$ such that $\E \widetilde {\C}_{w}(\tmu _{n}, \tLeb _{n};\Lambda_n) \leqslant cn$  for all $n \geqslant  1$ and similarly for $\tilde \nu_n$. \dy{By triangle inequality, we have that for some finite constant $c$,}
$$ \E \widetilde {\C}_{w}(\tmu _{n}, \tilde \nu_n ; \Lambda_n) \leq 4 \Big( \E \widetilde {\C}_{w}(\tmu _{n}, \tLeb _{n};\Lambda_n) + \E \widetilde {\C}_{w}(\tilde \nu_n, \tLeb _{n};\Lambda_n) \Big) \leq cn, $$
\dy{for all $n \geq 1$. Now using Proposition \ref{p:matching_limit}, there exist point processes $ \mu ' (\equlaw \mu), \nu ' ( \equlaw \nu)$ on \RL{the same probability space}  and a translation-invariant matching $T$ between $\mu'$ and $\nu'$ such that its  typical distance $X$ satisfies $ \mathbf E w(X) < \infty $.}

{Now we consider the cases $d = 1,2$. Here it suffices to prove \eqref{e:M0-inf_matching} i.e., there exists a translation-invariant matching from $\mu$ to $\nu$ such that the typical distance $X$ satisfies for some constant $c$,
\begin{equation}
\label{e:M0-inf_matching2}
\pr(X \geqslant  r) \leqslant c r^{-d/2} \sqrt{\sigma_{\mu}(r) + \sigma_{\nu}(r)}, \, \, r \geq 1.
\end{equation}
\dy{ The proof for $d =1,2$ follows from this bound and boundedness of $\sigma_\mu,\sigma_\nu$, which are implied by the integrability of the RPCM $\beta$.  To check this, verify by contradiction that for $ x,r>10, x^{d/2}\ln(x)^{-\gamma }>r$ implies $ x^{d/2}>cr\ln(r)^{\gamma }$. Now for $\gamma > 1$, substituting the above tail probability bound \eqref{e:M0-inf_matching2} along with boundedness of $\sigma_\mu,\sigma_\nu$ yields that
\begin{align*}
\mathbf E \left(
\frac{ X^{d/2}}{1 +  | \ln X | ^{ \gamma }}
\right) & \leqslant 10 + \int_{10}^{\infty} \mathbf{P}\Big( \frac{ X^{ \RL{d/2}}}{1 +  | \ln X | ^{ \gamma }} \geq r \Big) \md r \\
& \leqslant 10 + \int_{10}^{\infty} \mathbf P (X^{d/2} \geqslant cr\ln r^{ \gamma }) \md r <\infty \\
& \leqslant 10 + \int_{10}^{\infty} \frac{c}{r(1 + (\ln r)^{ \RL{\gamma }})} \md r < \infty,
\end{align*}
which concludes the proof.} 
The requirement of \eqref{e:M0-inf_matching} can be further reduced to exhibiting a $\Z^d$-translation invariant matching $T$ from \dy{$\nu$ to $\mu$} satisfying \eqref{e:M0-inf_matching2}.} \dy{Thanks to \eqref{def:typical_matching_dist}, \eqref{e:M0-inf_matching2} is implied by
\begin{equation}
\label{e:Exp_T_match}\mathbf{E}\Big( \dy{\nu} \{x \in [0,1]^d : \|T(x) - x\| \geqslant  r\} \Big) \leqslant  c r^{-d/2} \sqrt{\sigma_{\mu}(r) + \sigma_{\nu}(r)}, \, \, r \geq 1.
\end{equation}
}

 {The  construction in \cite[Theorem 1(ii)]{holroyd2009poisson} proceeds as follows. Let $ Z_k,k\geqslant 0$  be i.i.d. uniform random vectors in $\{0,1\}^d$ independent of $\mu,\nu$. We successively partition \RL{$\R^d$} into  
  randomly centered cubes of side-length $2^k$ \RL{ and randomly match points iteratively. More precisely, let    $ {U_k} = \sum_{i=0}^{k-1}2^iZ_i$  for $k \geq 0$ and define the   {\it $k$-level cubes} as the}
$$2^kz +U_{ k} +  [0,2^k]^d \RL{, z\in \mathbb{Z} ^{ d}}
.$$
At level $k=0$ all points are considered unmatched. At any level $k\geqslant 1$, for any $ k$-level cube $ C$, \RL{let $ m_{ C,\mu },m_{ C,\nu }$ the number of yet unmatched points  of $ \mu \cap C,\nu \cap C$ respectively. Form $ m_{ C}: = \min(m_{ C,\mu },m_{ C,\nu })$ pairs of matched points of  $\mu$ and $\nu$ within $ C$ using an arbitrary rule ; for example, we can use the stable marriage algorithm  as in \cite{hoffman2006stable}.  The points still unmatched  after this step are carried forward to the next level $ k + 1$ and the procedure is repeated.   Eventually, all points of either $ \nu$ or $ \mu$ are matched, all remaining points (i.e., the unmatched points) are said to be matched to $ \infty $. Denote by $ T$ the corresponding matching from $\nu \cup \{\infty\}$ to $\mu \cup \{\infty\}$. If $T$   satisfies the tail bound \eqref{e:Exp_T_match}, then letting $r \to \infty$, we can see that a.s. every point is matched to a finite distance, i.e. $T$ matches every $\nu$ point to a $\mu$ point, one can argue vice-versa as well; see \cite[Lemma 7]{holroyd2009poisson}.} Thus, we have that $T$ is a (perfect) matching between $\mu$ and $\nu$ and by construction $\Z^d$-invariant.  

Now we show the tail bound \eqref{e:Exp_T_match} for the partial matching $T$ as described above. \RL{The proof relies on the estimate 
$$ \mathbf{E}\Big(  {\nu} \{x \in [0,1]^d : \|T(x) - x\| \geqslant  \sqrt{d}2^k\} \Big) \leqslant  2^{-dk}\mathsf{E}S_{+}$$ 
where $S = (\nu(B) - \mu(B))$ for $B = [0,2^k)^d$ and recall that $S_+ = \max\{S,0\}$.  We refer the reader to the geometric proof of equation (11) in \cite{holroyd2009poisson} as it is exactly identical.} {Note that $\sqrt{d}2^k$ bounds the diameter of any $k$-level cube. In the above derivation, a crucial observation is that for  $x \in \nu$ to be unmatched or matched to a $\mu$-point at distance larger than $\sqrt{d}2^k$, it cannot be matched within the $k$-level cube $C$ containing $x$.    This can happen only if there are excess $\nu$ points than $\mu$ in the $k$-level cube $C$. Hence $x \in \nu$ has matching distance larger than $\sqrt{d}2^k$ only if $\nu(C) > \mu(C)$. Then relying on the mass transport principle and the fact that the random centering is independent of $\mu,\nu$, one can derive the above bound.}
Since $\mu,\nu$ have equal intensities, we can bound this via Cauchy-Schwarz and triangle inequality as follows.
$$ \mathbf{E}S_{+} \leqslant  \sqrt{\mathbf{E}(S^2)} \leqslant  \sqrt{2 \, (\, \var(\mu(B))+\var(\nu(B)) \, )} = \sqrt{2^{dk+1} \, ( \, \sigma_{\mu}(B) + \sigma_{\nu}(B) \, )},$$
where recall that  $\sigma_{\mu}(B) := |B|^{-1}\var(\mu(B))$ and similarly $\sigma_{\nu}(B)$. This gives the bound in \eqref{e:Exp_T_match} for $r = \sqrt{d}2^k$  and one can then extrapolate it to all $r$.}\\

\dy{ Now we justify the case of the shifted lattice $\mathbb{Z}^d + U$ with $U$ being uniformly distributed in $[0,1]^d$. Let $\nupl$ be the \RL{cloaked} stationary i.i.d. perturbed lattice in $d \geq 1$ as defined in Example \ref{ex:iidpertlattice}. Note that it has an integrable RPCM. We have proven that there exists a translation-invariant matching $T_1$ between $\mu$ and $\nupl$ (independently constructed on a probability space) such that its  typical distance $X$ satisfies $ \mathbf E w(X_1) < \infty $. By construction of $\nupl$, there exists a translation-invariant matching $T_2$ between $\mathbb{Z}^d + U$ and $\nupl$ such that its typical distance $X_2$ satisfies $X_2 \leq 2\sqrt{d}$ a.s.. Thus $T = T_1 \circ T_2$ gives a translation-invariant matching between $\mathbb{Z}^d + U$ to $\mu$ and using triangle inequality as above, we have that the typical distance $X$ satisfies $ \mathbf E w(X) < \infty $. }
\end{proof}
%

\begin{proof}[Proof of Theorem \ref{thm:hu}]
%

\noindent (1). \dy{Since our assumption on point processes $\mu,\nu$ for $d = 2$ is exactly the same as that in Theorem \ref{thm:Wbd-infinite-HU}(2), we obtain that $\E \widetilde W_2^2(\tmu _{n}, \tLeb _{n};\Lambda_n) \leqslant cn$  for all $n \geqslant  1$ and similarly for $\tilde \nu_n$. Now via triangle inequality as in Theorem \ref{thm:std}, we derive that for some finite constant $c$,
$$ \E \widetilde W_2^2(\tmu _{n}, \tilde \nu_{n};\Lambda_n) \leqslant cn, \, \forall n \geq 1.$$
Now using Proposition \ref{p:matching_limit}, we obtain that there exist point processes $ \mu ' (\equlaw \mu), \nu ' ( \equlaw \nu)$ on \RL{the same probability space}  and a translation-invariant matching $T$ between $\mu'$ and $\nu'$  whose  typical distance $X$ satisfies $ \mathbf E X^2 < \infty $.  Thus the conclusion follows for independent HU point processes $\mu,\nu$ with integrable RPCM and satisfying the logarithmic integrability condition.}. 

\dy{The case of shifted lattice $\Z^d + U$ can again be justified as at the end of the  proof of Theorem \ref{thm:std}. In particular, the stationary i.i.d. perturbed lattice $\nupl$ in Example \ref{ex:iidpertlattice} is hyperuniform and its RPCM satisfies the logarithmic integrability condition as well since it is compactly supported. Thus there exists a translation-invariant  matching $T_1$ between $\mu$ and $\nupl$ whose typical distance $X_1$ satisfies $\E X_1^2 < \infty$. Now using the canonical matching $T_2$ between $\Z^d + U$ and $\nupl$, we can justify existence of a translation-invariant matching $T : \Z^d + U \to \mu$ with typical distance $X$ satisfying $\E X^2 < \infty$.} \\

\noindent (2). \, For $d = 1$.  \RL{Assume that the number variance is bounded for $ \mu $ and $ \nu $}. Now using \eqref{e:M0-inf_matching2}, we obtain that there exists a translation-invariant matching $T$ such that the typical matching distance $X$ satisfies
$$ \pr(X \geqslant  r) \leqslant cr^{-1}, \, \, r \geq 1.$$
Notice that for $ r>1$, if $ X/(1 +  | \ln X | ^{ \gamma })\geqslant r$, then $ X\geqslant r$ \dy{and so $ X\geqslant r(1 + \ln r^{ \gamma })$. Now for $\gamma > 1$, substituting the above tail probability bound yields that}
\begin{align*}
\mathbf E \left(
\frac{ X}{1 +  | \ln X | ^{ \gamma }}
\right) & \leqslant 1 + \int_1^{\infty} \mathbf{P}\Big( \frac{ X}{1 +  | \ln X | ^{ \gamma }} \geq r \Big) \md r \\
& \leqslant 1 + \int_1^{\infty} \mathbf P (X\geqslant r(1 + (\ln r)^{ \gamma })) \md r <\infty \\
& \leqslant 1 + \int_1^{\infty} \frac{c}{r(1 + (\ln r)^{ \gamma })} \md r < \infty,
\end{align*}
\dy{which concludes the proof.}
\end{proof}

\section*{Acknowledgements} We are thankful to Rapha\"{e}l Butez, Manjunath Krishnapur, Thomas Lebl\'e and David Dereudre,  for sharing insight, interrogations and technical ideas about transport properties of point processes and the connection with perturbed lattices. The authors are immensely grateful to two anonymous referees for careful reading and numerous comments which have significantly improved the presentation. Thanks are also due to Jonas Jalowy whose comments alerted us about an issue in the proof of Proposition \ref{p:matching_limit}(2).
This work was initiated when DY visited Laboratoire MAP5 under CNRS fellowship and he thanks both the organizations for their support.  DY's research was also partially supported by SERB-MATRICS Grant MTR/2020/000470 and CPDA from the Indian Statistical Institute.

\bibliographystyle{abbrvnat}
\bibliography{transport_HU.bib}

\end{document}